\documentclass[12pt,reqno,a4paper]{amsart}
\usepackage{a4}
\usepackage{verbatim}

\title[Minimal Lagrangian Surfaces in $\CH^2$]{Minimal Lagrangian Surfaces in
$\CH^2$ and Representations of Surface Groups into $SU(2,1)$}
\author{John Loftin}
\author{Ian McIntosh}
\address{Department of Mathematics and Computer Science\\ Rutgers-Newark\\
Newark, NJ 07102, USA}
\email{loftin@rutgers.edu}
\address{Department of Mathematics\\ University of York\\
York YO10 5DD, UK}
\email{ian.mcintosh@york.ac.uk}
\date{\today}

\newcommand{\rp}{\mathbb{RP}}
\newcommand{\re}{\mathbb{R}}
\newcommand{\co}{\mathbb{C}}

\newcommand{\na}{\nabla}
\newcommand{\sfrac}[2]{{\textstyle \frac{#1}{#2}}}

\newcommand{\Z}{\mathbb{Z}}

\newcommand{\C}{\mathbb{C}}
\newcommand{\Ct}{\mathbb{C}^\times}

\newcommand{\R}{\mathbb{R}}

\renewcommand{\P}{\mathbb{P}}
\newcommand{\CP}{\mathbb{CP}}
\newcommand{\CH}{\mathbb{CH}}
\newcommand{\RH}{\mathbb{RH}}

\newcommand{\caD}{\mathcal{D}}

\newcommand{\caJ}{\mathcal{J}}
\newcommand{\caK}{\mathcal{K}}

\newcommand{\gl}{\mathfrak{gl}}

\newcommand{\fg}{\mathfrak{g}}
\newcommand{\fk}{\mathfrak{k}}
\newcommand{\fm}{\mathfrak{m}}

\newcommand{\fu}{\mathfrak{u}}

\newcommand{\Aut}{\operatorname{Aut}}

\newcommand{\Hom}{\operatorname{Hom}}
\newcommand{\Isom}{\operatorname{Isom}}

\renewcommand{\Re}{\operatorname{Re}}

\renewcommand{\Im}{\operatorname{Im}}

\newcommand{\Kah}{K\" ahler\ }
\newcommand{\Ad}{\operatorname{Ad}}
\newcommand{\ad}{\operatorname{ad}}
\newcommand{\tr}{\operatorname{tr}}

\newcommand{\la}{\langle}
\newcommand{\ra}{\rangle}

\newcommand{\g}[2]{\langle{#1},{#2}\rangle}
\newcommand{\II}{\mathrm{I\! I}}
\newcommand{\bz}{{\bar z}}
\newcommand{\diag}{\operatorname{diag}}

\newtheorem{thm}{Theorem}[section]
\newtheorem{prop}[thm]{Proposition}
\newtheorem{lem}[thm]{Lemma}
\newtheorem{cor}[thm]{Corollary}

\newtheorem{defn}[thm]{Definition}
\theoremstyle{remark}

\newtheorem{rem}{Remark}

\numberwithin{equation}{section}

\begin{document}
\begin{abstract}
We use an elliptic differential equation of \c{T}i\c{t}eica (or Toda) type to
construct a minimal
Lagrangian surface in $\CH^2$ from the data of a compact hyperbolic Riemann surface and a
cubic holomorphic differential. The minimal Lagrangian surface is equivariant
for an $SU(2,1)$ representation of the fundamental group. We use this data to
construct a diffeomorphism between a neighbourhood of the zero section in a
holomorphic vector bundle over Teichmuller space (whose fibres parameterise cubic
holomorphic differentials) and a neighborhood of the $\R$-Fuchsian representations
in the $SU(2,1)$ representation space. We show that all the
representations in this neighbourhood are complex-hyperbolic quasi-Fuchsian by
constructing for each a fundamental domain using an $SU(2,1)$ frame for the minimal
Lagrangian immersion: the Maurer-Cartan equation for this frame is the
\c{T}i\c{t}eica-type equation. A very similar equation to ours governs
minimal surfaces in hyperbolic 3-space, and our paper can be interpreted as an analog
of the theory of minimal surfaces in quasi-Fuchsian manifolds, as first studied by
Uhlenbeck.
\end{abstract}
\maketitle

\section{Introduction.}

The equation
\begin{equation}\label{eq:hyptoda}
\frac{\partial^2}{\partial z\partial\bar z}\log(s^2)=|Q|^2s^{-4}+s^2
\end{equation}
is satisfied by the conformally flat metric $2s^2|dz|^2$ of a minimal
Lagrangian surface in the complex hyperbolic plane $\CH^2$, where
$z=x+iy$ is a local conformal coordinate and $Q\,dz^3$ is a
holomorphic cubic differential. We can treat this equation either as
a local form or as an expression for the equations on the universal
cover of a compact surface $\Sigma$. In fact,  this equation is an
integrability condition: satisfying it is a
necessary condition for the existence of a minimal Lagrangian immersion of
$\Sigma$ into $\CH^2$.

There is also a coordinate invariant version. Fix a background
metric $h$ on a surface $\Sigma$. Then the universal cover of
$\Sigma$ admits a minimal Lagrangian immersion into $\CH^2$ with
metric $e^uh$ if it admits a smooth function $u:\Sigma\to\R$ and a
holomorphic cubic differential $U$ for which
\begin{equation}\label{eq:hypglobal}
\Delta_hu-16\|U\|_h^2e^{-2u} -2e^u -2\kappa_h=0
\end{equation}
where $\Delta_h,\kappa_h$ are respectively the Laplacian and
curvature, and $\|U\|_h$ is the norm on cubic differentials, all
with respect to $h$.

We prove the existence of global solutions to this equation (\ref{eq:hypglobal})
on any compact hyperbolic surface $\Sigma$ provided $U$ is sufficiently
small.  These equations are actually necessary and sufficient conditions for
the existence of a special frame $F:\tilde\Sigma\to SU(2,1)$ (called a
\emph{Legendrian frame})
for a minimal Lagrangian immersion $\varphi:\tilde\Sigma\to \CH^2$,
where $\tilde\Sigma$ denotes the universal cover if $\Sigma$. This frame
determines a flat $SU(2,1)$-bundle over $\Sigma$ whose holonomy
provides a representation of the fundamental group $\pi_1\Sigma$ into $SU(2,1)$
for which the map $\varphi$ is equivariant.

The latter part of the paper concerns properties of the representations we produce.
All of the representations we produce have zero Toledo invariant, as
they arise from Lagrangian surfaces  (the Toledo invariant characterises the
connected components of the representation space of surface groups into $SU(2,1)$
\cite{xia00}). When $U=0$, the
minimal Lagrangian surface is simply the canonical totally geodesic Lagrangian embedding
of $\RH^2$
 in $\CH^2$. The corresponding representation
takes values in $SO(2,1)\simeq PSL(2,\R)$ and is said to be $\R$-Fuchsian (it is a
Fuchsian representation). For $U$ small, we prove
the induced minimal Lagrangian surface is properly embedded into
$\CH^2$, and that the exponential map of the normal bundle of the
surface covers all of $\CH^2$. This allows us to construct a locally
finite fundamental domain for the $\pi_1\Sigma$ action on $\CH^2$,
simply by taking a bundle of Lagrangian planes normal to the immersion over a fundamental
domain on the minimal Lagrangian surface. As a consequence, each
representation we produce is \emph{complex-hyperbolic quasi-Fuchsian}, i.e.,
discrete, faithful, geometrically finite and totally loxodromic. To be precise, we prove the following
theorem (this is a restatement of Theorem \ref{main-thm} below).
\begin{thm} \label{main-thm-intro}
Let $S$ be a closed oriented surface of genus $g\ge2$. In the representation
space ${\rm Hom}(\pi_1S,SU(2,1))/SU(2,1)$ there is a neighborhood $\mathcal P$ of
the locus of $\re$-Fuchsian representations so that for all $\rho\in \mathcal P$,
\begin{itemize}
\item $\rho$ is complex-hyperbolic quasi-Fuchsian.
\item There is a natural identification of $\rho$ with a pair $(\Sigma,U)$
consisting of $\Sigma$ a marked conformal structure on $S$ and $U$ a small
holomorphic cubic differential on $\Sigma$. In particular, there is
submersion of $\mathcal P$ onto the Teichm\"uller space of $\re$-Fuchsian
representations, and a complex structure on $\mathcal P$.
\item There is a canonical $\rho$-invariant minimal Lagrangian embedding $\caD\subset\CH^2$
of the Poincar\'e disc, and an invariant normal projection of $\CH^2\to \caD$.
\end{itemize}
\end{thm}
There are several aspects to this construction which we consider to be valuable
and deserve further study.

First, it gives a holomorphic parameterisation for an open set of complex-hyperbolic quasi-Fuchsian
representations (a neighbourhood of the locus of all $\R$-Fuchsian representations) as a neighbourhood
of the zero section in a holomorphic vector bundle over Teichm\" uller space of rank $5g-5$.
It was already known from the work of Guichard \cite{guichard04} and Parker-Platis
\cite{parker-platis06} that the $\R$-Fuchsian locus possesses an open neighbourhood of complex-hyperbolic
quasi-Fuchsian representations (in fact, Guichard's result says that the space of
complex hyperbolic quasi-Fuchsian representations is open\footnote{
Misha Kapovich has also informed us that this result
was known earlier, and follows from the techniques used to address the case
of $SO(n,1)$; see e.g.\ Izeki \cite{izeki00}. }) but at this point in time not much is known
about how big this set is within the Toledo invariant zero component.
The fundamental domains we produce are in the end  similar to those of
Parker-Platis, as both consist of unions of Lagrangian planes, but with our data we get some
measure, through the norm of $U$, of how far we are away from the $\R$-Fuchsian locus.
Moreover, this parameterisation has an intriguing interpretation in terms of the Yang-Mills-Higgs
bundle description of representation space (see Remark \ref{rem:Higgs} below) which could help
explain how far this parameterisation can extend.

Second, our approach is analogous to the study of minimal surfaces in quasi-Fuchsian
hyperbolic $3$-manifolds initiated by Uhlenbeck in \cite{uhlenbeck83} (and continued in
\cite{taubes04,krasnov-schlenker07,wang09,huang-wang10}).
Indeed, one of the main goals of the study of surface group representations into $SU(2,1)$
is to find the extent to which the theory of quasi-Fuchsian representations
of surface groups extends to the complex hyperbolic case (for a recent survey, see
\cite{parker-platis10}).
The conformal factor of a minimal surface in $\mathbb{RH}^3$ solves
an equation analogous to \eqref{eq:hypglobal},
\begin{equation}\label{eq:quasi-fuchsian}
\Delta_hu-16\|V\|_h^2e^{-u} -2e^u -2\kappa_h=0,
\end{equation}
where $V$ a holomorphic quadratic differential.
It is known that for quasi-Fuchsian representations near enough to Fuchsian (called
almost Fuchsian) there is a unique invariant minimal surface in $\mathbb{RH}^3$.
On the other hand, there are quasi-Fuchsian representations for which
there are many minimal surfaces. Presumably, the complex-hyperbolic
representations we produce here are analogous to the almost Fuchsian case.
The solutions to \eqref{eq:hypglobal} we use are what we call
\emph{small} solutions (which means, when $\kappa_h=-1$, that the metric $g=e^uh$
has curvature bounds $-3/2\leq\kappa_g\leq -1$).
Provided $U$ is sufficiently small, \eqref{eq:hypglobal} has exactly one small solution
determined by $(\Sigma,h,U)$.
But it is possible that there are complex-hyperbolic quasi-Fuchsian
representations  far enough away from $\re$-Fuchsian to admit
multiple invariant minimal Lagrangian surfaces.



It is also worth noting that equation \eqref{eq:hypglobal} is one of several formally similar
equations which arise from surface geometries corresponding to different
real forms of $SL(3,\C)$, most of which have attracted attention in the
recent literature. These are all variations on the theme of \emph{\c{T}i\c{t}eica's
equation},
\[
u_{xy} + e^{-2u} -e^u=0.
\]
This hyperbolic equation corresponds to
nonconvex proper affine spheres in $\re^3$, and the symmetry group is
$SL(3,\R)$. The techniques of
integrating surfaces given solutions to equations of this type
originated with \c{T}i\c{t}eica's papers
\cite{tzitzeica08,tzitzeica09}. The more modern variants, which are distinct from
\eqref{eq:hypglobal}, are
\begin{eqnarray}\label{eq:CP2}
\Delta_hu-16\|U\|_h^2e^{-2u} + 2e^u -2\kappa_h&=&0,\\ \label{eq:hyp-aff-sph}
\Delta_hu+16\|U\|_h^2e^{-2u} - 2e^u -2\kappa_h&=&0,\\ \label{eq:ell-aff-sph}
\Delta_hu+16\|U\|_h^2e^{-2u} + 2e^u -2\kappa_h&=&0.
\end{eqnarray}
In each case $U$ is a cubic holomorphic differential and $e^uh$ is a metric for,
respectively: a minimal Lagrangian surface
in $\CP^2$ \eqref{eq:CP2}, where the isometry group is $SU(3)$ (see, for example,
\cite{mcintosh03,haskins-kapouleas07});
hyperbolic \eqref{eq:hyp-aff-sph} and elliptic \eqref{eq:ell-aff-sph} affine spheres in $\re^3$,
where the symmetry group is $SL(3,\R)$
(see \cite{wang91,simon-wang93}).
The latter two equations were also recently studied in order
to construct solutions to the
Monge-Amp\`ere equation $\det(\partial^2 u/\partial x^i\partial x^j)  =1$
on affine manifolds diffeomorphic to $\re^3$ minus the ``Y'' vertex of a graph
\cite{lyz05,lyz-erratum}. Equation (\ref{eq:hyp-aff-sph}) can also be
used to parameterise the Hitchin component of the representation
space of surface groups into $SL(3,\re)$
\cite{labourie97,labourie07,loftin01}.

Given Theorem \ref{main-thm-intro}, the next challenge is to
understand all the representations which can be obtained from
equivariant minimal Lagrangian immersions of the Poincar\' e disc.
All will have zero Toledo  invariant and therefore lie in the
same connected component of representation space. This will require
a greater understanding of the
solutions to equation (\ref{eq:hypglobal}) in the case where
$\|U\|_h$ is not ``small.''  Schoen-Wolfson's theory of mean
curvature flow of Lagrangian submanfolds in K\"ahler-Einstein
surfaces \cite{schoen-wolfson01} might be useful here.  Analogous
theories of surfaces which realise representations have been worked
out for some of the equations mentioned above.  For example, in the
case of equation (\ref{eq:quasi-fuchsian}) each quasi-Fuchsian
hyperbolic manifold admits at least one immersed minimal surface
(see Uhlenbeck \cite{uhlenbeck83}).  Cheng-Yau provide a similar
theory for equation (\ref{eq:hyp-aff-sph}) by showing that each
nondegenerate convex cone in $\re^3$ contains a hyperbolic affine
sphere invariant under any unimodular affine automorphisms of the
cone \cite{cheng-yau77,cheng-yau86}.

We would both like to thank John Parker for inspiring discussions. In addition, the first author would like to thank Bill Goldman, Misha Kapovich,
 and Anna Wienhard for stimulating conversations.

\medskip\noindent
\textbf{Notation.} For $u,w\in\C^n$ we use $u\cdot v$ to denote the standard (complex bilinear) dot
product, and set $\|u\|=\sqrt{u\cdot\bar u}$. We use $e_1,\ldots,e_n$ to denote the standard basis for
$\C^n$. For any non-zero $u\in\C^{n+1}$ we use $[u]\in\CP^n$ to denote the complex line it generates.

\section{Complex hyperbolic geometry.}

\subsection{Complex hyperbolic $n$-space.} Recall that complex hyperbolic $n$-space is the
complex manifold
\[
\CH^n = \{w=(w_1,\ldots,w_n)\in\C^n:\|w\|^2<1\}
\]
equipped with the Hermitian metric $\sum_{j,k=1}^nh_{j\bar k}dw_j\otimes d\bar w_k$ with components
\begin{equation}\label{eq:h}
h_{j\bar k} = \frac{1}{1-\|w\|^2}(\delta_{jk}+\frac{\bar w_jw_k}{1-\|w\|^2}).
\end{equation}
and \Kah form
\[
\omega = \frac{i}{2}\partial\bar\partial\log(1-\|w\|^2).
\]
With this metric $\CH^n$ has constant holomorphic sectional curvature $-4$.
We can embed $\CH^n$ in $\CP^n$ by
\begin{equation}\label{eq:embed}
\CH^n\to\CP^n;\quad w\mapsto [w,1]=[w_1,\ldots,w_n,1].
\end{equation}
Let $\la\ ,\ \ra$ denote the indefinite Hermitian form on $\C^{n+1}$ given by
\[
\la u,v\ra = u_1\bar v_1+\ldots +u_n\bar v_n - u_{n+1}\bar v_{n+1}.
\]
We see that \eqref{eq:embed} identifies $\CH^n$, as a manifold, with $\P W_-$, the space
of complex lines in $W_{-} = \{u\in\C^{n+1}:\la u,u\ra <0\}$.

Let $\pi:L\to\P W_-$ denote the tautological line bundle. We note that $W_-$ can be
identified with $L$ with its zero section removed.
Using the standard identification $T_zW_-\simeq\C^{n+1}$ we obtain a splitting
\[
TW_- = \mathcal V + \mathcal H
\]
where $\mathcal V=\ker(d\pi)$ and the horizontal subspace at $z$ is
\begin{equation}\label{horiz}
\mathcal H_z=\{u\in\C^{n+1}:\la u,z\ra = 0\}.
\end{equation}
On $\mathcal H$ the form $\la\ ,\ \ra$ is positive definite.
Further, this splitting is invariant for the $\Ct$ action along
fibres of $\pi$, since these fibres are the $\Ct$ orbits, and the
metric on $\mathcal H$ is also invariant. Therefore we obtain an
identification of $T\P W_-$ with $\pi_*H$, by assigning to each
tangent vector its horizontal lift. This equips $\P W_-$ with a \Kah
structure whose Levi-Civita connexion is the horizontal projection
of flat differentiation in $\pi_*\mathcal H\subset \P
W_-\times\C^{n+1}$. One can easily show that the \Kah structure $\P
W_-$ inherits from $\CH^n$ agrees with that obtained from $\pi_*
\mathcal H$, hence $\pi$ is a pseudo-Riemannian submersion.

\subsection{$\CH^n$ as a symmetric space.} As a symmetric space, $\CH^n$ is the non-compact
dual to $\CP^n$.  Since we will make use of
this for deriving the equations let us summarise the relevant facts.
The Lie group $G=U(n,1)$ of isometries for $\la\ ,\ \ra$
acts transitively on the pseudo-sphere
\[
S_-=\{v\in\C^{n+1}:\la v,v\ra=-1\} \subset W_-,
\]
and we will consider $S_-$ to be the $G$-orbit of $e_{n+1}$.
This action descends to a transitive action of $U(n,1)$ on $\CH^n$ by holomorphic isometries.
The isotropy group $K$ for this action is isomorphic to
$U(n)\times S^1$, and $\CH^n\simeq G/K$ as manifolds.

Let $q$ be the diagonal matrix $\mathrm{diag}(1,\ldots,1,-1)$ which represents the form $\la\ ,\ \ra$
(i.e., $\bar v^tqu = \la u,v\ra$). Define a involution $\sigma\in\Aut(\gl(n+1,\C))$ by
$\sigma(A) = qAq$. Then the Lie algebra of $G$ is
\[
\fg=\fu(n,1) = \{A\in\gl(n+1,\C): \sigma(A)=-\bar A^t\}.
\]
This involution $\sigma$ restricts to $\fg$ and provides the symmetric space decomposition
$\fg = \fk + \fm$
into $\pm 1$ eigenspaces of $\sigma$, with $\fk\simeq \fu(n)\times i\R$ and
$\fm\simeq \C^n$, the latter via
\[
\C^n\to\fm;\quad u\mapsto \begin{pmatrix} O_n & u\\ \bar u^t & 0 \end{pmatrix}
\]
where $O_n$ in the $n\times n$ zero matrix.

The coset map $G\to G/K$ is a principal
$K$-bundle. Using the right adjoint action of $K$ on $\fm$ we obtain
as associated vector bundle $[\fm]=G\times_K\fm$ which can be
identified with $T(G/K)$ (see, for example,
\cite[p6]{burstall-rawnsley90}).
The Hermitian metric on $G/K$ is obtained from the (unique up to scale)
$\Ad K$-invariant inner product on $\fm$ and
the map $G/K\to \P W_-;$ $gK\mapsto [ge_{n+1}]$ provides an isomorphism of \Kah
manifolds, $G/K\simeq\CH^n$.

\section{Lagrangian immersions in $\CH^2$.}

On $\C^3$, let $\g{\ }{\ }$, $W_-$ and $S_-$ be as above. For $u\in
S_-$, the tangent space
\[
T_uS_- =\{v\in \C^3:\Re\g{v}{u}=0\}
\]
contains the horizontal subspace $\mathcal H_u = \{v\in
T_uS_-:\g{v}{u}=0\}$. The form $\g{\ }{\ }$ is a Hermitian inner
product on $\mathcal H_u$ with real and imaginary components
\[
\g{\ }{\ } = g(\ ,\ ) + i\omega(\ ,\ ).
\]
These provide the Riemannian metric and the symplectic structure on
the horizontal bundle. The map $S_-\to \CH^2$ is a horizontal
isometry of Hermitian structures.

For $\caD$ the Poincar\' e disc, let $\varphi:\caD\to\CH^2$ be a Lagrangian immersion. It lifts
horizontally to a map $f:\caD\to S_-$ which is Legendrian (note the curvature of $L$ is proportional to $\omega$, and so its restriction vanishes for a Lagrangian immersion). This has a natural
$U(2,1)$-frame for $f$, which we describe below. We will show that the Maurer-Cartan equations for
this frame depend only on the induced metric, the mean curvature and a cubic differential.
In section \ref{sec:minimal} we will that see that when $\varphi$ is minimal the cubic
differential is holomorphic,
when $\varphi$ is also equivariant for a representation of a surface group into $PU(2,1)$,
the metric and cubic differential live on the quotient surface.

Using the conformal parameter $z$ on $\caD$ we can characterise $f$ as
Legendrian by the equations
\begin{equation}\label{eq:Leg}
\g{f_z}{f}=0=\g{f_\bz}{f}.
\end{equation}
A priori this only seems to force $f$ to be horizontal, but by
differentiating the first equation with respect to $\bz$, and the
second with respect to $z$, we find $\g{f_z}{f_z} =
\g{f_{\bz}}{f_{\bar z}}$, which implies the Legendrian condition
$\omega(f_x,f_y)=0$. Since $\g{\ }{\ }$ is positive definite on the
horizontal subspace we can write $|f_z| = \sqrt{\g{f_z}{f_z}} =
|f_{\bz}|$.

We will also assume that $\varphi$ is conformal, and hence $f$ is
horizontally conformal, i.e.,
\begin{equation}\label{eq:conf}
\g{f_z}{f_{\bz}}=0.
\end{equation}
Thus we obtain a global frame $F:\caD\to U(2,1)$ with columns
\begin{equation}\label{eq:Legframe}
F = \begin{pmatrix}f_1 & f_2 & f_3\end{pmatrix},\quad f_1 =
\frac{1}{|f_z|}f_z,\  f_2 = \frac{1}{|f_{\bz}|}f_{\bz},\ f_3=f.
\end{equation}
Now set $\alpha = F^{-1}dF$. We want to calculate the Maurer-Cartan
equations
\begin{equation}\label{eq:MC}
d\alpha + \alpha\wedge\alpha =0.
\end{equation}
Let $A = F^{-1}F_z$ and $B=F^{-1}F_{\bz}$, then
\begin{equation}\label{eq:MCform}
\alpha = Adz+Bd\bz
\end{equation}
This is a $\fu(2,1)$-valued $1$-form, i.e, $q\alpha q =
-\bar\alpha^t$, where $q=\diag(1,1,-1)$. It follows that $B = -q\bar
A^tq$ and the entries of the matrix $qA$ are $\g{(f_j)_z}{f_k}$ for
the $j^{th}$ column and $k^{th}$ row.

Set $s=|f_z|=|f_{\bz}|$ so that metric induced on $\caD$ by $\varphi$
is $2s^2|dz|^2$. We calculate
\begin{equation}\label{eq:Fz}
(f_1)_z  =  s^{-1}(f_{zz}-s_zf_1),\quad (f_2)_z  =
s^{-1}(f_{z\bz}- s_zf_2),\quad (f_3)_z  =  f_z.
\end{equation}
Now it is useful to have an expression for the second fundamental
form and the mean curvature of $\varphi$. Write $z=x+iy$ and define
an orthonormal basis for $T\caD\subset f^{-1}\mathcal H$ by
\[
E_1 = \frac{f_x}{|f_x|},\quad E_2 = \frac{f_y}{|f_y|}.
\]
Then $iE_1,iE_2$ span the normal bundle $T\caD^\perp$. Since
$S_-\to\CH^2$ is a horizontal isometry, the second fundamental form
of $\varphi$ is given by
\[
\II(X,Y) = \sum_{j=1}^2 g(XYf,iE_j)iE_j,\quad X,Y\in\Gamma(T\caD).
\]
(This follows from the fact mentioned above that the Levi-Civita
connection on $\CH^2$ is the projection of the flat connection on
$W_-\subset \C^3$.) Therefore, the mean curvature is
\[
H = \frac{1}{2} \sum_{j=1}^2
g(E_1^2f+E_2^2f,iE_j)iE_j=\frac{1}{2}(E_1^2f+E_2^2f)^\perp.
\]
Now $|f_x|^2 = |f_y|^2=2s^2$ and it is simple to check that
\[
g(f_{xx}+f_{yy},f_x)=0 = g(f_{xx}+f_{yy},f_y)
\]
and therefore
\[
(E_1^2f+E_2^2f)^\perp = \frac{2}{s^2}(f_{z\bz}+\g{f_{z\bz}}{f}f),
\]
taking care to observe that $\g{f}{f}=-1$. But
$\g{f_{z\bz}}{f}=-|f_z|^2$, therefore
\[
H = \frac{1}{s^2}f_{z\bz}-f = \frac{1}{2}\Delta_gf-f,
\]
where we have abused notation by letting $g$ stand for $\varphi^*g$
as well.

Using \eqref{eq:Fz} we write the quantities $\g{(f_j)_z}{f_k}$ as
they would appear in the matrix $qA$.
\begin{equation}\label{eq:qA}
\begin{array}{lll}
\g{f_{zz}}{f_z}s^{-2}-s_zs^{-1} & \g{f_{z\bz}}{f_z}s^{-2}  & s \\
\g{f_{zz}}{f_{\bz}}s^{-2} & \g{f_{z\bz}}{f_{\bz}}s^{-2}-s_zs^{-1} & 0\\
\g{f_{zz}}{f}s^{-1} & \g{f_{z\bz}}{f}s^{-1} & 0
\end{array}
\end{equation}
Many of these terms simplify. We have
\[
\g{f_{zz}}{f_z}s^{-2} = (\g{f_z}{f_z}_z - \g{f_z}{f_{z\bz}})s^{-2} =
2s_zs^{-1}-\g{f_z}{H}.
\]
Together, $\g{f_z}{f}=0$ and $\g{f_z}{f_{\bz}}=0$ imply
$\g{f_{zz}}{f}=0$, and therefore
$\g{f_{zz}}{f_{\bz}}=-\g{f_{zzz}}{f}$. Those identities also show
that for $Q=\g{f_{zzz}}{f}$ the quantity $Q\,dz^3$ is a cubic
differential.

Finally, $\g{f_{z\bz}}{f_{\bz}}=s^2\g{H}{f_{\bz}}$, and therefore
we deduce that
\begin{equation}\label{eq:A}
A = \begin{pmatrix}
s_zs^{-1} - \g{f_z}{H} & \g{H}{f_z} & s\\
-Qs^{-2} & \g{H}{f_{\bz}} -s_zs^{-1} & 0\\
0& s&0
\end{pmatrix}
\end{equation}
It follows that
\begin{equation}\label{eq:B}
B=\begin{pmatrix}
-s_{\bz}s^{-1} + \g{H}{f_z} & {\bar Q}s^{-2} & 0\\
-\g{f_z}{H} & -\g{f_{\bz}}{H} +s_\bz s^{-1} & s\\
s & 0& 0
\end{pmatrix}
\end{equation}
At this point the following observation is useful.
\begin{lem}
Let $\sigma_H = \omega(H,df)$, which is the mean curvature $1$-form
for $\varphi$ (pulled back by the lift $f$). Then
$i\sigma_H=\g{H}{df}= \tfrac{1}{2}\tr\alpha$.
\end{lem}
\begin{proof}
First observe that since $g(H,df)=0$, we clearly have $i\sigma_H =
\g{H}{df}$. Now observe that
\begin{eqnarray*}
\g{H}{f_\bz}& =&  \tfrac{1}{2}\g{H}{f_x+if_y}\\
& = & \tfrac{1}{2}\g{H}{f_x} -\tfrac{i}{2}\g{H}{f_y}\\
& = & \tfrac{i}{2}\omega(H,f_x) +\tfrac{1}{2}\omega(H,f_y)
\end{eqnarray*}
and therefore $\g{H}{f_z} = -\g{f_\bz}{H}$ (and, equally,
$\g{f_z}{H}=-\g{H}{f_\bz}$). Now
\begin{eqnarray*}
\tr\alpha& = & (\g{H}{f_\bz} - \g{f_z}{H})dz + (\g{H}{f_z}-\g{f_\bz}{H})d\bz\\
&=& 2\g{H}{f_\bz d\bz + f_z dz}.
\end{eqnarray*}
\end{proof}
It follows that the frame $F$ (up to a constant scaling) satisfies $\det(F)=1$ if and only if
$\varphi$ is a minimal immersion, i.e., $H=0$. Otherwise to make the
frame an $SU(2,1)$ frame we must divide $F$ by $\det(F)^{1/3}$,
which is the cube root of the Lagrangian angle function $\exp(2i\int
\sigma_H):\caD\to S^1$. Notice that $d\tr\alpha=0$ (which follows from
the Maurer-Cartan equations) implies that $\sigma_H$ is a closed
$1$-form: this is to be expected since $\CH^2$ is a K\"ahler-Einstein manifold.

Now that we have the form of $\alpha$, we can state the
Maurer-Cartan equations.
\begin{prop}
The form $\alpha$ satisfies the Maurer-Cartan equations if and only
if all three of the following equations hold
\begin{eqnarray}\label{eq:3MC}
\frac{\partial^2}{\partial z\partial\bz}\log(s^2)& = &s^2+|Q|^2s^{-4} -\frac{s^2}{2}|H|^2\\
d\sigma_H&=&0\\
\label{eq:Qbz}
Q_\bz s^{-4}& = &-(\g{H}{f_\bz}s^{-2})_z.
\end{eqnarray}
\end{prop}
Here we have used
\[
\g{H}{f_z}= \frac{1}{2}(\omega(H,f_x)-i\omega(H,f_y)) =
\frac{s}{\sqrt{2}}(g(H,iE_1)-ig(H,iE_2)),
\]
to deduce that
\[
|\g{H}{f_z}|^2 = \frac{s^2}{2}|H|^2.
\]
Note that, via the isomorphism $T\CH^2\simeq G\times_K\fm$, the differential $d\varphi$ corresponds
to $\Ad F\cdot\alpha_{\fm}$ and the first equation in \eqref{eq:3MC} is essentially the Gauss equation
for $\varphi$.

\section{Minimal Lagrangian immersions in $\CH^2$.}\label{sec:minimal}
For a minimal Lagrangian surface $H=0$, and so
for the Maurer-Cartan form $\alpha = A\,dz + B\,d\bz$,
\begin{equation}\label{eq:AB}
A = \begin{pmatrix}
s^{-1}s_z & 0 & s\\ -Qs^{-2}& -s^{-1}s_z & 0 \\ 0 & s & 0
\end{pmatrix},\quad
B =
\begin{pmatrix}
-s^{-1}s_{\bar z} & \bar Qs^{-2} & 0 \\ 0 & s^{-1}s_{\bar z} & s\\s & 0&0
\end{pmatrix}
\end{equation}
Moreover, the integrability conditions (\ref{eq:3MC}) and
(\ref{eq:Qbz}) become \eqref{eq:hyptoda} and $Q_{\bz}=0$, so that
$Q$ is a holomorphic cubic differential.

Now consider the global theory on a Riemann surface. Let
$(\Sigma,h)$ be a closed Riemann surface of genus at least $2$ with metric $h$.
Fix a uniformisation $\caD\to\Sigma$, and express the metric $h$
over $\caD$ as $\gamma|dz|^2$.

Suppose $\varphi:\caD\to\CH^2$ is minimal Lagrangian and
$\bar\rho$-equivariant for a representation $\bar\rho:\pi_1\Sigma\to PSU(2,1)$.
The circle bundle $S_-\to\CH^2$
is the unit subbundle of the tautological bundle $L$ and as a bundle with connexion
$L^3\simeq K_{\CH^2}$ (viewing
$\CH^2\subset\CP^2$). Let $f$ be a global horizontal
section of the flat $S^1$-bundle $\varphi^{-1}S_-$.
For the same reasons as the case of $\CP^2$ \cite{hunter-mcintosh10} the
mean curvature $1$-form $\sigma_H$ is the connexion
$1$-form for this flat connexion on $\varphi^{-1}K_{\CH^2}$. Since $\varphi$
is minimal this connexion has trivial holonomy, and so the
holonomy group for the contact structure connexion on
$\varphi^{-1}S_-$ is either trivial or $\Z_3$. Hence $f$ is equivariant for a representation
$\rho:\pi_1\Sigma\to SU(2,1)$ which lies over $\bar\rho$.
It induces a metric $|df|^2=e^uh=2s^2|dz|^2$ and a
cubic holomorphic differential $U=Qdz^3=\g{f_{zzz}}{f}dz^3$ on $\caD$ which are both
$\rho$-invariant. According to the previous section, they satisfy
\[
\frac{\partial^2}{\partial z\partial\bar z}\log(\gamma) +
\frac{\partial^2u}{\partial z\partial\bar z}
 = \tfrac12 e^u\gamma +4|Q|^2e^{-2u}\gamma^{-2}.
\]
The Laplacian and curvature with respect to $h$ are given by
\[
\Delta_h = \frac{4}{\gamma}\frac{\partial^2}{\partial z\partial\bar z},\quad
\kappa_h = -\frac{2}{\gamma}\frac{\partial^2}{\partial z\partial\bar z}\log(\gamma),
\]
so that the equation becomes
\[
\Delta_hu-2\kappa_h = 2e^u+16\frac{|Q|^2}{\gamma^3}e^{-2u}.
\]
But in these coordinates $\|U\|_h^2 = |Q|^2/\gamma^3$ and therefore
we obtain \eqref{eq:hypglobal}.

Conversely, suppose we have a triple
$(\Sigma,U,u)$, where $u:\Sigma\to\R$ is a global solution of \eqref{eq:hypglobal}.
Let $\alpha$ be the Maurer-Cartan form over $\caD$ given by \eqref{eq:AB}.
Fix any base point $z_0\in\caD$ and let
$F$ be the unique $SU(2,1)$ frame for which $F^{-1}dF=\alpha$ and $F(z_0)=I$. It is easy to see that a
holomorphic change of coordinates $w(z)$ results in the change of frame
\begin{equation}\label{eq:change_of_frame}
F \mapsto F c_{zw},\quad
c_{zw} =
\begin{pmatrix}
z_w/|z_w|&0&0\\0&\overline{z_w}/|z_w|&0\\ 0&0&1
\end{pmatrix}.
\end{equation}
The quantities $z_w/|z_w|$ and $\overline{z_w}/|z_w|$ are, respectively, the transition functions
for the unit circle subbundle in $T^{1,0}\Sigma$ and and its inverse. Hence
$\alpha$ determines a principal $SU(2,1)$-bundle $P\to\Sigma$ equipped with a flat connexion $\theta$
whose expression in the local frame $F$ is the Maurer-Cartan form $\alpha$ above. The holonomy of this flat
connexion determines a representation $\rho$ up to conjugacy in $SU(2,1)$, hence we obtain an well-defined
element of $\Hom(\pi_1\Sigma,SU(2,1))/SU(2,1)$. Moreover, by \eqref{eq:change_of_frame} the last
column of $F$
is independent of the local coordinate and determines a $\rho$-equivariant map $f:\caD\to S_-$, which
is minimal Legendrian and the horizontal lift of a $\rho$-equivariant minimal Lagrangian map
$\varphi:\caD\to\CH^2$ with induced metric $e^uh$ and cubic holomorphic differential $U$.

In summary, we have proven the following theorem. As above, we assume we have fixed a
uniformisation $\caD\to\Sigma$ and holomorphic coordinate $z$ on $\caD$.
\begin{thm}\label{thm:Legframe}
Let $(\Sigma,h)$ be a compact Riemannian surface of genus at least $2$
and $U$ a globally holomorphic cubic differential on $\Sigma$
for which there exists a solution $u:\Sigma\to\R$ to \eqref{eq:hypglobal}.
Let $\alpha$ be given by \eqref{eq:AB}, with $e^uh = 2s^2|dz|^2$ and $U=Qdz^3$. Then
we obtain a minimal Lagrangian immersion
$\varphi:\caD\to\CH^2$ by integration of the equations
$F^{-1}dF=\alpha$, $F:\caD\to SU(2,1)$.
The map $\varphi$ is uniquely determined, up to isometries of
$\CH^2$, by the data $(\Sigma,e^uh,U)$, and is equivariant with
respect to a holonomy representation $\rho:\pi_1\Sigma\to SU(2,1)$
lying in the conjugacy class corresponding to the flat $SU(2,1)$-bundle
$(\Sigma,P,\theta)$ determined by $\alpha$.

Conversely, a minimal Lagrangian immersion $\varphi: \caD\to\CH^2$
equivariant with respect to a representation $\bar\rho:\pi_1\Sigma\to PSU(2,1)$
determines a metric $e^uh$ and a holomorphic cubic differential
$U$ on $\Sigma$ which satisfy \eqref{eq:hypglobal}. Up to conjugacy,
the representation $\rho$ this data determines lies over $\bar\rho$.
\end{thm}
\begin{rem}\label{rem:Higgs}
The flat bundle $(\Sigma,P,\theta)$ has a corresponding Yang-Mills-Higgs description in
the sense of Corlette \cite{corlette88}. For we can split $\alpha$ into $\alpha_\fk+\alpha_\fm$ according
to the reductive decomposition $\fg=\fk+\fm$. Since the transition functions \eqref{eq:change_of_frame}
lie in the isotropy subgroup $K$ there is a corresponding splitting of $\ad P$ into
$\ad K$-invariant subbundles $V_\fk+V_\fm$ and $\theta$ determines a $K$-connexion
$D$ on each subbundle together with a section $\Psi$ of $V_\fm$. In our local frame $D=d+\ad\alpha_\fk$
and $\Psi=\alpha_\fm$. These satisfy the
Yang-Mills-Higgs equations, which assert that $D+\ad\Psi$ is flat and $D^*\Psi=0$.
This data can also be treated as holomorphic Higgs bundle data using the type
decomposition of $1$-forms and connexions (see, for example, Xia \cite{xia00}). From our
point of view $D$ is the Levi-Civita connexion on $\varphi^{-1}T\CH^2$ and $\Psi$ is the
differential $d\varphi$. Now recall that according to
Corlette for \emph{any} choice of conformal structure on $\Sigma$ every reductive representation
in $\Hom(\pi_1\Sigma,SU(2,1))/SU(2,1)$ corresponds to Yang-Mills-Higgs data. Indeed, what Corlette proves is that for
every choice of conformal structure on the smooth surface $\Sigma$ and every reductive representation
$\rho:\pi_1\Sigma\to
SU(2,1)$ there is a $\rho$-equivariant harmonic map $\varphi:\caD\to\CH^2$. The data we have satisfies
the extra conditions corresponding to $\varphi$ being conformal harmonic and Lagrangian: in terms of
$(D,\Psi)$ conformality is the condition that $\g{\Psi^{1,0}}{\Psi^{0,1}}=0$ while $\varphi$
is Lagrangian when $\Im\g{\Psi}{\Psi}=0$. Thus the conformal Lagrangian conditions tie the conformal
structure of $\Sigma$ to the data $(D,\Psi)$ and impose conditions on $\Psi$ which
could be thought of as putting it into a normal form.
\end{rem}



\section{Global solutions.}\label{sec:global}

In this section, we find global solutions to (\ref{eq:hypglobal}) on
a compact hyperbolic surface, by using similar techniques to those
developed in \cite{lyz05,lyz-erratum}.

\begin{thm} \label{exist-sol}
Let $\Sigma$ be a compact Riemann surface equipped with a metric $h$
of constant Gaussian curvature $-1$.  If $U$ is a nonzero
holomorphic cubic differential on $\Sigma$ which satisfies
$$\max_\Sigma \|U\|_h^2 \le \frac 1{54}$$
for $\|\cdot\|_h$ the metric on cubic differentials determined by
$h$, then there is a smooth solution $u$ to the equation of
\c{T}i\c{t}eica type
\begin{equation} \label{tz-eq}
\Delta u - 16\|U\|_h^2 e^{-2u} - 2 e^u + 2 = 0
\end{equation}
on $\Sigma$, where $\Delta$ is Laplacian with respect to $h$.
\end{thm}
\begin{proof}
Let $H(u) =\Delta u - 16\|U\|_h^2 e^{-2u} - 2 e^u + 2$. The existence
of a smooth solution to $H(u)=0$ follows if we can construct sub-
and super-solutions $s,S$ on $\Sigma$ satisfying
$$ s\le S, \qquad H(s)\ge0, \qquad H(S) \le0 $$
(see e.g.\ Schoen-Yau \cite{schoen-yau94}, Proposition V.1.1). Such
a solution satisfies $S\ge u \ge s$.
Let $S = 0$. Then clearly $H(S) = -16\|U\|_h^2e^{-2S}\le0$.

Similarly, let $s=C$ for $C$ a negative number, and let $M =
\max_\Sigma  \|U\|_h^2$. Then compute $$H(s)  = - 16 \|U\|_h^2
e^{-2C} - 2 e^{C} + 2 \ge -16M e^{-2C} - 2e^{C} + 2.$$  Consider the
function $f(C) = -16M e^{-2C} - 2e^{C} + 2$ for $M>0$.  Then
$f(0)<0$ and $f\to-\infty$ as $C\to-\infty$.  The only critical
point of $f$ occurs at
$$C=C_{\max} =  \frac13 \log (16M).$$  Compute $$f(C_{\max}) = -
6\cdot2^{\frac13}\cdot M^{\frac13} + 2 \ge 0 \qquad
\Longleftrightarrow \qquad M \le \frac 1{54}.$$ So if $M$ satisfies
this bound, $C_{\max}\le \log(2/3)<0$ and $f(C_{\max})\ge0$, which
shows that $s=C_{\max}$ is a subsolution.
\end{proof}
To obtain a corresponding uniqueness for these solutions, we must introduce the following constraint.
\begin{defn}\label{U-small-bound}
On a hyperbolic Riemann surface $(\Sigma,h)$ equipped with a
holomorphic cubic differential $U$, we call a function $u$
\emph{small} when it provides the bound $16\|U\|_h^2e^{-3u}\leq 1$, in other
words, when
\begin{equation}\label{small}
u \ge \frac13\log(16\max_\Sigma\|U\|_h^2).
\end{equation}
\end{defn}
\begin{rem}\label{small-curvature}
The geometric significance of this constraint, and the reason for the somewhat
counter-intuitive use of the word \emph{small} to describe a function bounded \emph{below},
is that $u$ is a small solution of \eqref{tz-eq} if and only if the metric $g=e^uh$ has
\emph{small curvature}
\[
-3/2\leq \kappa_g\leq -1.
\]
The upper bound is true for any solution to \eqref{tz-eq}: it is the lower bound which
equates to \eqref{small}.
\end{rem}
We will mainly be
interested in small solutions to (\ref{tz-eq}).
Let
\[\theta(x,u)=-16\|U\|^2_he^{-2u} - 2e^u +2,
\]
so that (\ref{tz-eq}) becomes $\Delta u +\theta(x,u)=0$.  A simple computation implies
\begin{lem} \label{theta-dec}
For a small function $v$,  $$\frac{\partial \theta}{\partial
u}(x,v)\le0.$$
\end{lem}
\begin{prop} \label{unique-sol}
Given $(\Sigma,h,U)$, there is at most one small solution $u$ to
(\ref{tz-eq}).  The solutions produced by Theorem \ref{exist-sol}
above are small.
\end{prop}
\begin{proof}
Suppose that $v,w$ are two small solutions of (\ref{tz-eq}).
We will use the
comparison principle and Lemma \ref{theta-dec} show that $v=w$.
By assumption, $\Delta v+\theta(x,v)=0$, $\Delta w
+\theta(x,w)=0$. Consider the path of functions $u_t=tv+(1-t)w$.
Then a standard computation shows $v-w$ satisfies
$$
\Delta(v-w) =
-\left(\int_0^1 \frac{\partial \theta}{\partial u}(x,u_t)\,dt
\right)(v-w).
$$
Therefore, $v-w$ satisfies the linear elliptic equation $L(v-w)=0$,
for
 \begin{equation}\label{def-L}
 L\phi = \Delta\phi + \left(\int_0^1\frac{\partial \theta}{\partial
u}(x,u_t)\,dt\right)\phi\equiv \Delta\phi+c\phi.
 \end{equation}
Since $v$ and $w$ are both small, $u_t$ is small for all
$t\in[0,1]$, and Lemma \ref{theta-dec} shows $c\le0$. At this point,
the strong maximum principle (Theorem 3.5 in
\cite{gilbarg-trudinger}) applies, and we may conclude that $v-w$ is
either constant on $\Sigma$ or has no nonnegative maximum. If $v-w$
is constant, it is easy to show the constant must be 0. Therefore,
we may conclude $v-w$ has no positive maximum---i.e., $v-w\le0$ on
$\Sigma$. By symmetry, $w-v\le0$ on $\Sigma$ also, and so we must
have $v=w$.

That the solutions produced in Theorem \ref{exist-sol} are small is
evident from the proof.
\end{proof}
\begin{lem} \label{improve-small}
Assume $M=\max_\Sigma\|U\|_h^2\le\frac1{16}$.  Then any small
solution $u$ to (\ref{tz-eq}) satisfies
$$\chi_M \le u \le 0,$$ where $\chi_M$ is the largest real root
of $f(C) = -16M e^{-2C} - 2e^{C} + 2$ if $M\le\frac1{54}$ and
$\chi_M=\log(16M)/3$ if $\frac1{54}< M \le \frac1{16}$.
\end{lem}
\begin{proof}
If $M\le\frac{1}{16}$ then clearly $w=0$ is small.
Set $u_t = tu+(1-t)w$ and compute as in the proof of Proposition
\ref{unique-sol}
\begin{eqnarray*}
H(w)&=& \Delta
w+\theta(x,w)= -4\|U\|^2_h e^{-2w}\le0, \\
\Delta (u-w) &=& -\theta(x,u) +\theta(x,w)-H(w) \\
&\ge& -[\theta(x,u)-\theta(x,w)] \\
&=& - \left(\int_0^1 \frac{\partial \theta}{\partial u}(x,u_t)\,dt
\right) (u-w).
\end{eqnarray*}
Therefore, $u-w$ satisfies $L(u-w)\ge0$ for $L\phi=\Delta\phi +
c\phi$ as in (\ref{def-L}), with $c\le0$ by Lemma \ref{theta-dec}.
Again, the strong maximum principle implies either that $u-w$ is
constant (which is easily ruled out except for the case $U=0$,
$u=0$) or that $u-w$ has no nonnegative maximum. Therefore
$u\le w$ on all $\Sigma$.

Similar reasoning shows that $\chi_M$ is a lower bound for any small
solution $u$. When $M\le\frac1{54}$,  $\chi_M$ is small. The proof
of Theorem \ref{exist-sol} shows that $f$ achieves a nonnegative
maximum value at its only critical point $C_{\rm max}<0$ when
$M\le\frac1{54}$.  On the other hand, $f(0)<0$.  The
Intermediate Value Theorem implies $\chi_M\ge C_{\rm max}$, which is
equivalent to the definition for $\chi_M$ to be small.

Moreover,
$$H(\chi_M) = \Delta\chi_M - 16\|U\|^2_h e^{-2\chi_M} - 2e^{\chi_M} + 2
\ge f(\chi_M)=0.$$ This implies that for $u_t = t\chi_M + (1-t)u$
and $L\phi=\Delta\phi+c\phi$ as in (\ref{def-L}), $L(\chi_M-u)\ge0$
with $c\le0$. Then the strong maximum principle implies $\chi_M\le
u$ on all $\Sigma$.
\end{proof}

\begin{cor} \label{u-limit}
Fix $(\Sigma,h)$. As $U\to0$, the unique small solution $u=u_U$ to
(\ref{tz-eq}) approaches $0$ uniformly.
\end{cor}
\begin{proof}
As $M\to0$, $\chi_M\to 0$.
\end{proof}
\begin{thm}
Given a closed hyperbolic Riemann surface $(\Sigma,h)$, the family
of small solutions $u=u_U$ to (\ref{tz-eq}) is smoothly varying in
$U$ for $$U\in\mathcal U_{1/54} = \left\{ U : M=\max_\Sigma
\|U\|^2_h \le \frac1{54} \right\}.$$
\end{thm}
\begin{proof}
We use the continuity method. Consider $$H(u,U) = \Delta u -
16\|U\|^2_h e^{-2u} -2e^u +2$$ for $U\in\mathcal U_{1/54}$ and $u\in
C^{2,\alpha}(\Sigma)$.  Then it is straightforward to check that $H$
is a Fr\'echet differentiable map from $C^{2,\alpha}\times \mathcal
U_{1/54} \to C^{0,\alpha}$. In order to use the Implicit Function
Theorem (for $U$ in the interior of $\mathcal U_{1/54}$), we must
check that the partial differential $$\frac{\delta H}{\delta u} \!:
\eta \mapsto \Delta \eta +32\|U\|_h^2 e^{-2u}\eta - 2 e^u \eta$$ has
a continuous inverse from $C^{0,\alpha}$ to $ C^{2,\alpha}$.  This
follows from checking the kernel of $\frac{\delta H}{\delta u}$
vanishes, which is true by the assumption that $u$ is small and the
maximum principle.  Thus there is a family of solutions $u_U$ for
each $U$ in a neighborhood of each $U_0$ in the interior of
$\mathcal U_{1/54}$. These solutions are still small by continuity
and the improved bound in Lemma \ref{improve-small} (since $\chi_M >
\frac13\log(16M)$). Then Proposition \ref{unique-sol} allows us to
identify these solutions with the ones already produced in Theorem
\ref{exist-sol}.

For good measure, the closedness part of the continuity method
follows from Lemma \ref{improve-small}, which shows that $u_U$ and
$\Delta u_U$ are uniformly in $L^p$ for any $p<\infty$. Then the
elliptic theory shows $u_U\in L^p_2$ uniformly.  Sobolev embedding
then gives uniform bounds in $C^{1,\alpha}$, and further
bootstrapping implies uniform $C^{2,\alpha}$ bounds of $u_U$  as $U$
varies. Ascoli-Arz\'ela allows us to take limits to show closedness.

The variation is smooth by standard elliptic theory.
\end{proof}
\begin{rem} We do not expect the bound $\|U\|_h^2\le \frac1{54}$ to be
 sharp as a condition for the existence of solutions.
\end{rem}
\begin{prop}
Let $\Sigma$ be a closed Riemann surface of genus $g\ge2$ equipped
with a hyperbolic metric $h$ as above. If $U$ is a holomorphic
cubic differential on $\Sigma$ which is large in the sense that
$$ \int_\Sigma |U|^{\frac23} > \frac{2\pi\sqrt[3]4}3\,(g-1),$$
then there is no solution to (\ref{tz-eq}) on $\Sigma$.
\end{prop}
\begin{proof}
Let $u$ be a solution to (\ref{tz-eq}). Integrate (\ref{tz-eq}) and
use Gauss-Bonnet to find
\begin{equation} \label{linear-con}
16\int_{\Sigma} \|U\|_h^2 e^{-2u} \, dV_h + 2 \int_{\Sigma} e^u\,
dV_h = 8\pi(g-1)
\end{equation}
for $dV_h$ the volume form of the hyperbolic metric.  H\"older's
inequality shows
\begin{equation} \label{int-u-lower-bound} \int_\Sigma |U|^{\frac23}
= \int_\Sigma \|U\|_h^{\frac23} \,dV_h \le \left(\int_\Sigma
\|U\|_h^2 e^{-2u}\, dV_h \right)^{\frac13} \left( \int_\Sigma e^u
dV_h \right)^{\frac23}.
\end{equation}
If we denote $$A = \int_\Sigma \|U\|_h^2 e^{-2u}\, dV_h, \qquad B =
\int_\Sigma e^u \, dV_h,$$ we can maximise $AB^2$ for $A,B>0$
subject to the constraint $8A+B = 4\pi(g-1)$ to find that $$AB^2 \le
\frac {32}{27} [\pi(g-1)]^3.$$ Then (\ref{linear-con}) and
(\ref{int-u-lower-bound}) prove the contrapositive of the
proposition.
\end{proof}

\section{Solutions with zero cubic differential.} \label{U-zero}

In this section, we study minimal Lagrangian surfaces in $\CH^2$
corresponding to $U=Q\,dz^3=0$. First of all, consider solutions to
(\ref{tz-eq}) if $U=0$. By the maximum principle, we have
\begin{lem}
If $\Sigma$ is a closed Riemann surface equipped with hyperbolic
metric $h$ and cubic differential $U=0$, then the unique solution to
(\ref{tz-eq}) is $u=0$.
\end{lem}

On the upper half-plane $\{x+iy:y>0\}$, consider $Q=0$. In this
case, the metric corresponding to $s = \frac1{y\sqrt2}$ solves
(\ref{eq:hyptoda}). The connexion matrices $A,B$ satisfy
\begin{eqnarray*} A &=&\left(\begin{array}{ccc} \frac{i}{2y} &
0 & \frac1{y\sqrt2} \\ 0 & -\frac i{2y} & 0 \\
0 & \frac1{y\sqrt2} & 0 \end{array} \right), \\
 B &=&
\left(\begin{array}{ccc} \frac i{2y} & 0&0\\ 0 & -\frac i{2y} &
\frac 1{y\sqrt2} \\ \frac 1{y\sqrt2} & 0 & 0
\end{array} \right), \\
F^{-1}F_x = A+B &=& \left(\begin{array}{ccc} \frac iy & 0 & \frac
1{y\sqrt2} \\ 0 & -\frac iy & \frac 1{y\sqrt2} \\ \frac 1{y\sqrt2} &
\frac 1{y\sqrt2} & 0 \end{array} \right) \equiv \frac Ly,\\
F^{-1}F_y = iA-iB &=& \left(\begin{array}{ccc}0&0& \frac i{y\sqrt2} \\
0&0& -\frac i{y\sqrt2} \\ -\frac i{y\sqrt2} & \frac i{y\sqrt2} & 0
\end{array}\right) \equiv \frac Ky.
\end{eqnarray*}
To solve the  initial value problem in $y$, let $y=e^t$ to find
$F^{-1}F_t = K.$ It is straightforward to integrate these equations
to find a fundamental solution of the initial-value problem for any
path from $(0,1)$ to $(x,y)$:
$$ \exp(Lx) \cdot \exp(Kt) = \exp(Lx) \cdot \exp(K\log y).$$ This formula
follows from the Maurer-Cartan equations, which show the fundamental
solution is independent of the choice of path. Thus we may integrate
along the piecewise-linear path from $(0,1)$ to $(x,1)$ to $(x,y)$.

\renewcommand{\arraystretch}{1.3}
It is convenient to take the initial condition  $$F_0 = (f_1 \,\,\,
f_2\,\,\, f_3) = \left(\begin{array} {ccc} \frac1{\sqrt2} &
\frac 1{\sqrt2} &0 \\ -\frac i{\sqrt2} & \frac i{\sqrt2} & 0 \\
0&0&1
\end{array} \right).$$ (This $F_0\notin SU(2,1)$, but we still use
it to ensure $f$ is real below. The factor $\det F_0=i$ is irrelevant upon projecting $W_-\to \CH^2$
in any case.) So the solution is $F= F_0 \cdot \exp (Lx) \cdot \exp
(K\log y)$, whose last column $f_3=f$ is given by $$f = \left(
\begin{array}{c}\frac xy \\ \frac{-1+x^2+y^2}{2y} \\
\frac{1+x^2+y^2}{2y} \end{array} \right).$$ Note $f$ parameterises
the upper component of the real hyperboloid $$(\mbox{Re}\,u_1)^2 +
(\mbox{Re }u_2)^2 - (\mbox{Re }u_3)^2=-1$$ in $\re^3\subset\co^3$, and the immersion $[f]$
into $\mathbb{CH}^2$ is the standard immersion of $\mathbb {RH}^2
\subset \mathbb{CH}^2$, which is of course minimal Lagrangian.
\renewcommand{\arraystretch}{1}

\section{Fundamental domains in $\mathbb {CH}^2$.}

We now show, at least for  $U$ near 0, that the minimal Lagrangian
surface produced by a small solutions $u$ to (\ref{tz-eq})
determines a fundamental domain for induced action of $\pi_1\Sigma$
on $\mathbb{CH}^2$.  Recall the notation of the Legendrian $SU(2,1)$
frame

$$ F=(f_1\,\,f_2\,\,f_3), \qquad f_1= \frac{f_z}{|f_z|}, \quad f_2 =
\frac{f_{\bar z}}{|f_{\bar z}|}, \quad f_3 = f.$$

At a point $p$ in $\tilde \Sigma$ the universal cover of
$\Sigma$, we may choose coordinates in $\mathbb C^3$ so that
\begin{equation}\label{frame-id} f_1 = \left(\begin{array}{c}1\\0\\0\end{array} \right), \qquad
 f_2 = \left(\begin{array}{c}0\\1\\0\end{array} \right), \qquad
f_3 = \left(\begin{array}{c}0\\0\\1\end{array} \right).
\end{equation}
We may also choose a conformal normal coordinate $z$ so that at $p$,
$z=0$ and $\psi=\psi_z=\psi_{\bar z}=0$ for the affine metric
$e^\psi|dz|^2$. In terms of $s=|f_z|=|f_{\bar z}| = \tfrac{1}{\sqrt{2}} e^{\psi/2}$,
this means $s=1/\sqrt{2}$, $s_z=s_{\bar z}=0$ at $z=0$. Moreover, we may
rotate $z$ so that at $z=0$, $Q\in[0,\infty)$. Under these assumptions, at
$z=0$,
\begin{equation} \label{F-evolve}
F_z=A = \left( \begin{array}{ccc} 0&0&1/\sqrt{2} \\-2Q&0&0 \\ 0&1/\sqrt{2}&0
\end{array} \right), \qquad F_{\bar z} =B = \left(\begin{array}{ccc} 0&2Q&0
\\ 0&0&1/\sqrt{2} \\ 1/\sqrt{2}&0&0 \end{array} \right).
\end{equation}

The tangent plane to $M=\pi(f(\tilde \Sigma))$ at $p$ is spanned by $\pi_*(f_x) =
\pi_*(f_z+f_{\bar z})=\pi_*(f_1+f_2)$ and $\pi_*(f_y) = \pi_* (if_z
- i f_{\bar z}) = \pi_*(if_1-if_2),$ for $\pi\!:W_-\to
\mathbb{CH}^2$ the projection. So the Lagrangian copy of $\mathbb
{RH}^2$ tangent to $M$ in $\mathbb{CH}^2$ can be described by
$$\mathcal T =\{\pi[(a+ib)f_1 + (a-ib)f_2 + c f_3] : a^2+b^2<\sfrac12c^2 \}.$$
This explicit description of the tangent space allows us also to
describe the totally geodesic Lagrangian plane normal to $\mathcal
T$ (the image of the normal space under the exponential map) as
\begin{equation}\label{normal-plane}
\mathcal N = \{\pi[(ia-b)f_1 + (ia+b)f_2 + cf_3] : a^2 + b^2
<\sfrac12c^2 \}.
\end{equation}
This is because the normal vectors
to a Lagrangian tangent plane are determined by the action of the
complex structure $J$ on tangent vectors.

\begin{thm} \label{immersion-criterion}
Let $\Sigma$ be a compact Riemann surface equipped with a cubic
differential $U$ and a solution $u$ to (\ref{tz-eq}).  For the disc
$$\mathcal D=\{(a,b)\in\re^2:a^2+b^2<\sfrac12\}$$
and $\tilde \Sigma$ the universal cover of $\Sigma$, consider the map $\Phi\!:
\tilde\Sigma\times\mathcal D \to \mathbb{CH}^2$ given by
$$\Phi(z,a,b) = \pi[(ia-b)f_1(z) + (ia+b)f_2(z) + f_3(z)].$$
Then $\|U\|_g = e^{-\frac32u}\|U\|_h \le 1/2$ on all of $\Sigma$
if and only if $\Phi$ is an immersion.
\end{thm}
\begin{proof}
To calculate when $\Phi$ is an immersion, we may work at a
point $z=0$, and make the coordinate assumptions
(\ref{frame-id}-\ref{F-evolve}) above. Set $w=b+ia$, and let $W$ be the vector
$(-\bar w,w,1)$, then $\Phi = \pi(FW)$ and
\[
d(FW) = AWdz+BWd\bar z+f_2dw -f_1d\bar w.
\]
Write $FW = (Z_1,Z_2,Z_3)$. Then in inhomogeneous coordinates
$\Phi=(Z_1/Z_3,Z_2/Z_3)$, and its differential has the form
\[
d\Phi = \frac{1}{Z_3^2}(Z_3dZ_1-Z_1dZ_3, Z_3dZ_2-Z_2dZ_3).
\]
At $z=0$, using the assumptions above, we have
\begin{eqnarray*}
(Z_1,Z_2,Z_3)& = &(-\bar w,w,1)\\
(dZ_1,dZ_2,dZ_3)& =&
(\tfrac{1}{\sqrt{2}}dz + 2Qwd\bar z-d\bar w,
2Q\bar wdz+\tfrac{1}{\sqrt{2}}d\bar z+dw,\tfrac{1}{\sqrt{2}}(wdz-\bar wd\bar z)).
\end{eqnarray*}
Thus
\begin{equation}\label{dPhi}
d\Phi = \begin{pmatrix}
\tfrac{1}{\sqrt{2}}(1+|w|^2)dz +(2Qw-\tfrac{1}{\sqrt{2}}\bar w^2)d\bar z -d\bar w\\
(2Q\bar w-\tfrac{1}{\sqrt{2}}w^2)dz + \tfrac{1}{\sqrt{2}}(1+|w|^2)d\bar z +dw
\end{pmatrix}
\end{equation}
Thus $d\Phi$ will be injective provided these two $1$-forms and their complex
conjugates are linearly independent. This happens
when the following matrix is invertible:
\[
\begin{pmatrix}
\tfrac{1}{\sqrt{2}}(1+|w|^2) & 2Qw-\tfrac{1}{\sqrt{2}}\bar w^2 & 0 &-1\\
 2Q\bar w-\tfrac{1}{\sqrt{2}} w^2 & \tfrac{1}{\sqrt{2}}(1+|w|^2) & -1 & 0\\
 2Q\bar w-\tfrac{1}{\sqrt{2}} w^2 & \tfrac{1}{\sqrt{2}}(1+|w|^2) & 1 & 0\\
\tfrac{1}{\sqrt{2}}(1+|w|^2) & 2Qw-\tfrac{1}{\sqrt{2}}\bar w^2 & 0 &1
\end{pmatrix}
\]
The determinant of this is
\begin{eqnarray*}
\caJ & = & 4(\tfrac12(1+|w|^2)^2-|2Qw-\bar w^2/\sqrt{2}|^2)\\
&=& 2(-8Q^2|w|^2+2\sqrt{2}Q(w^3+\bar w^3) +2|w|^2+1).
\end{eqnarray*}
We must establish conditions on $Q\geq 0$ for
which $\caJ$ is non-zero for all $|w|^2<\tfrac12$. Since $\caJ>0$ at $w=0$ we must
ensure $\caJ$ stays positive. Let us write
\[
\tfrac12 \caJ = -8\alpha Q^2 +2\sqrt{2}\beta Q + 2\alpha +1,
\]
for $\alpha = |w|^2$ and $\beta = w^3+\bar w^3$. By considering this as
a polynomial in $Q$, we see that it is positive for $0\leq Q<Q_0$, where $Q_0$
is the single positive root:
\[
Q_0 = \frac{1}{2\sqrt{2}}\left(\frac{\beta}{2\alpha}
+\sqrt{\left(\frac{\beta}{2\alpha}\right)^2+\frac{1}{\alpha}+2}\right).
\]
The condition $|w|^2<\tfrac12$ is equivalent to
\[
0\leq\alpha < \frac{1}{2},\quad
-\sqrt{\alpha}\leq \frac{\beta}{2\alpha}\leq \sqrt{\alpha}.
\]
It is straightforward to check that the infimum of $Q_0$ over $|w|^2\le1/2$ is only attained
at the boundary values $\alpha =1/2 $, $\beta=-1/\sqrt{2}$, at which $Q_0=1/2$. Thus
$Q\leq 1/2$ ensures that $d\Phi$ is invertible at $z=0$ for all $w\in\caD$.
On the other hand, whenever $Q> 1/2$ an examination of \eqref{dPhi} shows that
$\partial\Phi/\partial x=0$ at $w=-Q/2\sqrt{2}\in\caD$, so $\Phi$ cannot be an
immersion.

Finally, we note that $\|U\|_g=Q$ at $z=0$, therefore $\|U\|_g\leq 1/2$
uniformly over $\Sigma$ if and only if $\Phi$ is an immersion.
\end{proof}

\begin{prop}
There is a constant $\kappa$ so that if
$\|U\|_h<\kappa$ on all of $\Sigma$, then $\Phi$ is a diffeomorphism
from $\tilde \Sigma\times \mathcal D \to \mathbb{CH}^2$. Moreover,
the natural continuous extension $\bar \Phi \!: \tilde\Sigma \times
\bar {\mathcal D} \to \overline{\mathbb{CH}^2}$ is injective.
\end{prop}
\begin{proof}
For $U$ near 0, note that the estimates above show that $\|U\|_h$
and $\|U\|_g = e^{-\frac32u}\|U\|_h$ are equivalent norms.
Therefore, Theorem \ref{immersion-criterion} above shows there is a
bound on $\sup\|U\|_h$ which implies $\Phi$ is an immersion.
On the other hand, $\Phi$ is a proper map if and only if $[f]$ is a
proper map from $\tilde \Sigma\to \mathbb{CH}^2$.  This is because
$\Phi$ corresponds to the exponential map on the normal bundle in
the direction transverse to the image of $[f]$, and so must be
proper in that direction. Proposition \ref{proper} below shows there
is a constant bound $k$ so that if $\|U\|_h<k$, then $[f]$ is
proper.

Therefore, there is a bound $\kappa$ so that if $\|U\|_h<\kappa$,
then $\Phi$ must be a proper immersion from $\tilde\Sigma\times
\mathcal D \to \mathbb{CH}^2$.  Proper immersions of connected
manifolds of the same dimension are covering maps (this can be
proved by the same techniques as the proof of the Stack of Records Theorem,
Problem 1.4.7 in \cite{guillemin-pollack}).  Thus $\Phi$ a
diffeomorphism, since its domain is simply connected.

To show $\bar \Phi$ is injective as well, note that the proof of
Theorem \ref{immersion-criterion} shows that if $\|U\|_g < 1/2$,
then $\bar\Phi$ is an immersion of manifolds with boundary. The
injectivity of $\Phi$ implies $\bar\Phi$ is injective also.
\end{proof}

\begin{cor} \label{fund-domain}
Corresponding to each surface produced in the previous proposition,
there is a fundamental domain $\mathcal F$ in $\mathbb {CH}^2$ for
the representation of $\pi_1\Sigma$ into $SU(2,1)$. Let
$\bar{\mathcal F}$ be the closure of $\mathcal F$ in
$\overline{\mathbb {CH}^2} \subset \mathbb{CP}^2$.  There are only a
finite number of $\gamma\in\pi_1\Sigma$ satisfying $\gamma\cdot \bar
{\mathcal F} \cap \bar {\mathcal F} \neq \emptyset$.
\end{cor}
\begin{proof}
We discuss below in Section \ref{rep-sec}  the induced
representation of $\pi_1\Sigma$ into $SU(2,1)$ from the point of
view of principal bundles.

Consider a fundamental domain for the action of $\pi_1\Sigma$ on
$\tilde\Sigma$, and then consider the portion of the total space of
the normal bundle of the embedded minimal Lagrangian surface over
this domain.

The last statement of the corollary follows from the injectivity of
$\bar\Phi$ and the corresponding fact for the fundamental domain on
the surface $\Sigma$.
\end{proof}

\section{Properness of the Immersion.}
\begin{prop} \label{proper}
There is a constant $k>0$ so that if $\|U\|_h<k$, then $[f]$ is
proper.
\end{prop}
\begin{proof}
First of all, by the construction of $\mathbb{CH}^2$ above, note
that $f\in S_- = \{v\in\co^3 : \langle v,v\rangle = - 1\}$ implies
that $$[f] \to \partial \mathbb{CH}^2 \quad \Longleftrightarrow
\quad \|f\|_E \to \infty$$ for $\|f\|_E$ the Euclidean norm on $\co^3$.
Therefore, $[f]$ is proper if and only if $\|f\|_E$ is unbounded along
any path to infinity in the universal cover $\tilde\Sigma$.  In
terms of suitable coordinates, we will show that $\|f\|_E$ has to grow
exponentially.

The proof proceeds by treating the developing map for $f$ as a
perturbation of the developing map in the case of $U=0$ with the
background hyperbolic metric (as in Section \ref{U-zero} above). The
key estimate involves an ODE system of form $X_t = (C+D(t))X$, where
$C$ is an explicit constant matrix and $D(t)$ is small enough and
bounded in absolute value.

\renewcommand{\arraystretch}{1.15}

Identify the universal cover $\tilde\Sigma$ of the Riemann surface
with the upper half-plane $\{z=x+iy:y>0\}$.  As above in Section
\ref{U-zero}, for our Legendrian frame $F$,
$$ F^{-1}F_y = iA -iB = \left( \begin{array}{ccc} is^{-1}(s_z + s_{\bar z})
& -i\bar Q s^{-2} &is \\ -iQs^{-2} & -is^{-1}(s_z +s_{\bar z}) & -is
\\ -is & is & 0
\end{array} \right),$$ where $U=Q\,dz^3$ and $2s^2|dz|^2 = e^uh$ for
$h=|dz|^2/y^2$ the hyperbolic metric.  Therefore,
$$F^{-1}F_y = \left(\begin{array}{ccc} \frac i2 \,u_x & -2i \bar Q y^2e^{-u}
& \frac{i}{\sqrt{2}}e^{\frac u2}y^{-1} \\ -2iQ y^2e^{-u} & -\frac i2 u_x & -\frac{i}{\sqrt{2}}e^{\frac
u2} y^{-1} \\ -\frac{i}{\sqrt{2}}e^{\frac u2} y^{-1} & \frac{i}{\sqrt{2}}e^{\frac u2} y^{-1}
&0\end{array} \right).$$

All of the terms of $F^{-1}F_y$ are on the order of $y^{-1}$: This
is obvious for the terms in the third row and column. As for $u_x$,
note $\|\na u\|_h = y\sqrt{u_x^2+u_y^2} \ge y|u_x|$.  By compactness
of $\Sigma$, there is a uniform bound on $\|\na u\|_h$ (improved by
Proposition \ref{bound-u-z} below), and thus there is a bound of the
form $|u_x| \le Cy^{-1}$. On the other hand, $|Q|$ transforms as a
section of $|K^3|$ over $\Sigma$ (where $K$ represents the canonical
bundle). Since $y^{-3}$ is an invariant section of $|K^3|$, we have
have e.g.\ $|-2iQy^2e^{-u}|$ is bounded by  $C' y^{-1}$.

In fact, we have better bounds on the entries in $F^{-1}F_y$ as
$U\to0$: $F^{-1}F_y$ equals
$$ \frac1y \left[\left( \begin{array}{ccc} 0&0&
\frac i{\sqrt2} \\ 0&0& -\frac i{\sqrt2} \\ - \frac i{\sqrt2} &
\frac i{\sqrt2} &0\end{array} \right) + \left(
\begin{array}{ccc} iu_xy & -2i\bar Qy^3e^{-u} & \frac{i}{\sqrt{2}}(
e^{\frac u2} - 1) \\ -2iQy^3e^{-u} & -iu_xy
&\frac{i}{\sqrt{2}}(1 -e^{\frac u2})\\ \frac{i}{\sqrt{2}}(1 -e^{\frac u2}) &
\frac{i}{\sqrt{2}}( e^{\frac u2} - 1) &0
\end{array} \right) \right].$$
If we write the second matrix as $\tilde G$, then Corollary
\ref{u-limit} above and Proposition \ref{bound-u-z} below show that
the maximum of the entries of $\tilde G$ go to zero as
$\sup_\Sigma\|U\|_h$ goes to 0.

We also change coordinates $t=\log y$ to show
$$F^{-1}F_t = \left(
\begin{array}{ccc} 0&0& \frac i{\sqrt2} \\ 0&0& -\frac i{\sqrt2} \\
- \frac i{\sqrt2} & \frac i{\sqrt2} &0\end{array} \right) + \tilde G,$$
in which the  constant matrix can be diagonalised with
eigenvalues $-1,0,1$. In fact,

$$\left( \begin{array}{ccc} i&-i& \sqrt2 \\ 1&1&0 \\ -i &i&\sqrt2\end{array} \right)\left(
\begin{array}{ccc} 0&0& \frac i{\sqrt2} \\ 0&0& -\frac i{\sqrt2} \\
- \frac i{\sqrt2} & \frac i{\sqrt2} &0\end{array} \right)
\left(
\begin{array}{ccc} -\frac i4 &\frac12 & \frac i4 \\ \frac i4 &
\frac12& -\frac i4 \\
\frac 1{2\sqrt2} & 0 & \frac1{2\sqrt2}\end{array} \right) = \left(
\begin{array}{ccc} -1&0&0 \\ 0&0& 0
\\ 0&0 &1\end{array} \right).$$
Let $Y=FP$ for $P$ the change of frame matrix listed third above.
Then
\begin{equation} \label {Y-eq}  Y^{-1}Y_t = \left(\begin{array}{ccc} -1&0&0\\0&0&0\\0&0&1
\end{array} \right) + G,
\end{equation}
where the conjugated matrix $G=P^{-1}\tilde G P$  satisfies the same
sort of sup norm estimates on its entries that $\tilde G$ does.

For initial conditions, we follow the model case in Section
\ref{U-zero} above by choosing at $(t,x)=(0,0)$,
$$F_0 = \left(\begin{array}{ccc} \frac1{\sqrt2} & \frac1{\sqrt2} &0
\\ - \frac i{\sqrt2} & \frac i{\sqrt2}  &0 \\ 0&0&1
\end{array}\right), \qquad Y_0 = \left(\begin{array}{ccc} 0&
\frac1{\sqrt2} & 0 \\ -\frac 1{2\sqrt2} & 0 & \frac 1{2\sqrt2} \\
\frac 1{2\sqrt2} & 0 &\frac 1{2\sqrt2}
\end{array}
\right).$$

\renewcommand{\arraystretch}{1}

Let $X =(x_1\,\,\,x_2\,\,\,x_3)= X(t)\in \co^3$ represent the bottom
row of $Y$, so that
$X(0)=(\frac1{2\sqrt2}\,\,\,0\,\,\,\frac1{2\sqrt2})$. Let $\tilde
x=(x_1,x_2)$, and let $G=(g_{ij}(t))$ with $|g_{ij}(t)|<\delta$.
Let $\epsilon>0$. First, we use (\ref{Y-eq}) to show that there are  $\gamma>0$ and $k>1$ so that $|x_3(\epsilon)|-k|\tilde
x(\epsilon)|>\gamma$ as long as $\delta$ is small, and
$k,\gamma$ depend only on $\delta,\epsilon$.  This follows from the
fact that, for a fixed $\epsilon>0$, the solution to the linear initial value problem $$ \dot X
= K(t)X, \qquad X(0)=\left(\sfrac1{2\sqrt2} \,\,\,\, 0
\,\,\,\,\sfrac1{2\sqrt2}\right)$$ on the interval
$t\in[0,\epsilon]$ varies continuously in $C^0([0,\epsilon])$ as
$K(t)$ varies in $C^0([0,\epsilon])$. This in turn follows by
inspection of the Picard iterates. Then note that if $K(t)= {\rm
diag}(-1,0,1)$, the solution $\Phi =
\left(\frac1{2\sqrt2}e^{-t}\,\,\,\, 0 \,\,\,\, \frac1{2\sqrt2}e^t
\right).$ Since our $X$ is $C^0$-close to $\Phi$, this ensures that
the third component of $X(\epsilon)$ is larger than the first two
components of $X(\epsilon)$.

In particular, $\hat X(t)=X(t+\epsilon)$ satisfies the hypotheses of
Proposition \ref{ode-infinity} below for $k>1$. Therefore, we can
choose a $\delta>0$ so that if $|g_{ij}(t)|<\delta$ for
$G=(g_{ij})$, then
$$|x_3(t)| - k|\tilde x(t)| \ge (|x_3(\epsilon)| - k|\tilde x(\epsilon)|)
e^{C(t-\epsilon)}\ge\gamma e^{C(t-\epsilon)}$$
for a constant
$C=C(\delta)>0$. The element $f_{33}$ of $F$ satisfies
$f_{33}=x_1\sqrt2 + x_3\sqrt2$, and so
$$|f_{33}(t)| \ge \sqrt2(|x_3(t)|-|x_1(t)|)\ge
\sqrt2(|x_3(t)|-k|\tilde x(t)|)\ge \sqrt2\,\gamma
e^{C(t-\epsilon)}.$$

For $y=e^t$, we have
$$ |f_{33}(t)| \ge \sqrt2\,\gamma e^{-C\epsilon}y^C.$$
But $f_{33}$ is the third component of the position vector $f$ of
the embedding, and so $\|f\|_E\to\infty$ along the path in the upper
half plane $\{iy:y\to\infty\}$.  In terms of the Poincar\'e disc
model, if $w=\frac{iz+1}{z+i}$, then along the radial path $\{ir : r
\to 1^-\}$, $$\|f(ir)\|_E \ge \sqrt2\,\gamma e^{-C\epsilon}
\left(\frac{1+r} {1-r} \right)^C.$$ But for any radial path
$\{e^{i\theta}r:r\to1^-\}$, the same estimates hold, since we may
reduce to the same problem by rotating both $w$ in the disc and $f$
in $\co^3$. Therefore, for any $w$ in the Poincar\'e disc, we have
$$\|f(w)\|_E \ge \sqrt2\,\gamma e^{-C\epsilon} \left( \frac{1+|w|}{1-|w|} \right)^C,$$
and so $f$ is a proper map into $\co^3$.
Therefore, $[f]$ is a proper map into $\mathbb{CH}^2$.
\end{proof}
\begin{prop} \label{ode-infinity}
If $X=(x_1\,\,\,x_2\,\,\,x_3)$ is a $\co^3$-valued solution to the
ODE system
$$ \frac {dX}{dt} = X\left[ \left(\begin{array}{ccc} -1&0&0 \\ 0&0&0
\\ 0 & 0&1 \end{array} \right) + G \right],$$
where $G=(g_{ij})=(g_{ij}(t))$ satisfies
$|g_{ij}|\le\delta<1/(4\sqrt2+3)$, then there are positive constants
$k=k(\delta)$ and $C=C(\delta)$ so that if the initial conditions
satisfy
$$|x_3(0)| > k|\tilde x(0)|,$$
for $\tilde x = (x_1,x_2)$, then for all $t>0$,
$$ |x_3(t)|-k|\tilde x(t)| \ge (|x_3(0)|-k|\tilde x(0)|)e^{Ct}.$$
$k$ is a continuous, decreasing function of $\delta$, with
$k(1/(4\sqrt2+3)) = 1$ and $k\to\infty$ as $\delta\to0$.
\end{prop}
\begin{proof}
Use the Cauchy-Schwartz Inequality to estimate
\begin{eqnarray*}
\frac d{dt}\, |x_3| &=& \frac{ \dot x_3 \bar x_3 + x_3\dot{\bar x}_3} {2|x_3|} \\
& =& \frac{2\mbox{Re}[(1+g_3^3)x_3\bar x_3 + g_3^2x_2\bar x_3 + g_3^1x_1\bar x_3] } {2|x_3|}\\
&\ge& (1-\delta)|x_3| -\delta\sqrt2|\tilde x|.
\end{eqnarray*}
Similarly, we may compute
$$\frac d{dt} |\tilde x| \le \delta\sqrt2|x_3| + 2\delta|\tilde x|.$$
Therefore, we have
\begin{equation}\label{d-dt}
\frac d{dt}\left(|x_3| - k|\tilde x|\right) \ge
(1-\delta-k\delta\sqrt2)|x_3| - (\delta\sqrt2+2k\delta)|\tilde x|.
\end{equation}
Now by our bound on $\delta$, we may choose $k$ to be the larger
root of
$$k^2-\left(\frac{1-3\delta}{2\delta \sqrt2} \right)k + 1 =0,$$
and also choose $$ C = 1-\delta-k\delta\sqrt2>0.$$ Then we have
$$-(\delta\sqrt2 + 2k\delta) = -Ck,$$
which, together with
(\ref{d-dt}), implies
$$ \frac d{dt} (|x_3|-k|\tilde x|) \ge C(|x_3| -k|\tilde x|).$$
Since the initial value of $|x_3|-k|\tilde x|$ is assumed to be
positive, the differential inequality shows
$$|x_3(t)|-k|\tilde x(t)| \ge (|x_3(0)|-k|\tilde x(0)|) e^{Ct}.$$
It is easy to check that $\lim_{\delta\to0}k=\infty$.
\end{proof}

\begin{prop} \label{bound-u-z}
Let $\Sigma$ be a compact Riemann surface equipped with a conformal
hyperbolic metric $h$ and a holomorphic cubic differential $U$. Let
$u$ be a small solution to (\ref{tz-eq}). Then there is a constant
$C$ depending only on $m=\sup e^{-2u}\|\na (\|U\|)^2\|$ so that
$\|\na u \|\le C$. As $U\to 0$,  $m\to0$ and $C\to0$.
\end{prop}
\begin{proof}
In local coordinates, write $h  = \gamma|dz|^2$. Let $v = \gamma
\,u_z u_{\bar z} = \frac14\|\na u\|^2$, and let $p$ be a maximum
point of $v$. Choose a local coordinate $z$ so that at $p=\{z=0\}$,
$\gamma_z(0)=\gamma_{\bar z}(0)=0$ and $\gamma(0)=1$. The condition
that $h$ is hyperbolic is then  $\gamma_{z\bar z}(0) = \frac12$.

Compute
\begin{eqnarray*} v_z &=& \gamma_z u_z u_{\bar z} + \gamma
u_{zz} u_{\bar z} + \gamma u_z u_{z\bar z}, \\
v_{\bar z} &=& \gamma_{\bar z} u_z u_{\bar z} + \gamma u_{z\bar z}
u_{\bar z} + \gamma u_z u_{\bar z \bar z},\\
v_{z \bar z} &=& \gamma_{z \bar z} u_z u_{\bar z} + \gamma_z u_{z
\bar z} u_{\bar z} + \gamma_z u_z u_{\bar z\bar z} + \gamma_{\bar z}
u_{zz} u_{\bar z} + \gamma u_{zz\bar z} u_{\bar z} + \gamma u_{zz}
u_{\bar z \bar z} \\
&&{}+ \gamma_{\bar z} u_z u_{z\bar z} + \gamma u_{z\bar z}u_{z\bar
z} + \gamma u_z u_{z\bar z\bar z}.
\end{eqnarray*}
At the maximum point,  $\na v= 0$ implies
\begin{equation}
\label{grad-v} u_{zz}u_{\bar z} + u_z u_{z\bar z} = 0 = u_{z\bar z}
u_{\bar z } +u_z u_{\bar z \bar z},
\end{equation}
while $v_{z\bar z} \le 0$ implies
$$ \sfrac12 u_z u_{\bar z} + u_{zz\bar z}u_{\bar z} + u_{zz}u_{\bar
z \bar z} + u_{z\bar z}u_{z\bar z} + u_z u_{z\bar z\bar z}\le 0.$$
This becomes, by (\ref{grad-v}),
\begin{equation}
\label{lap-v-ineq}\sfrac12 u_z u_{\bar z} + 2u_{z\bar z}u_{z\bar z}
+u_{zz\bar z} u_{\bar z} + u_z u_{z\bar z\bar z}\le 0.
\end{equation}
Now (\ref{tz-eq}) implies that
$$u_{z\bar z} = 4Q\bar Q \gamma^{-2} e^{-2u} + \sfrac12\gamma e^u - \sfrac12 \gamma,$$
where the cubic differential $U=Q\,dz^3$. Since
$\gamma = 1 + O(|z|^2)$, we compute at $p$
\begin{eqnarray*}
u_{z\bar z} &=& 4Q\bar Q e^{-2u} + \sfrac12 e^u -\sfrac12, \\
u_{zz\bar z} &=& 4Q_z \bar Q e^{-2u} -8Q\bar Qe^{-2u}u_z + \sfrac12 e^uu_z, \\
u_{z\bar z \bar z} &=& 4Q\bar Q_{\bar z} e^{-2u} -8Q\bar Q e^{-2u}
u_{\bar z} + \sfrac12 e^u u_{\bar z}.
\end{eqnarray*}
Therefore, (\ref{lap-v-ineq}) becomes
\begin{eqnarray*} 0&\ge& \sfrac12 u_zu_{\bar z} +
2(4Q\bar Qe^{-2u}+\sfrac12 e^u-\sfrac12)^2 + 4\bar Q e^{-2u} Q_z u_{\bar z} +
4Qe^{-2u} \bar Q_{\bar z} u_z \\ &&{}+(-16Q\bar Q e^{-2u}+e^u)u_zu_{\bar z} \\
&\ge &\sfrac12 u_zu_{\bar z} + 4\bar Q e^{-2u} Q_z u_{\bar z} +
4Qe^{-2u} \bar Q_{\bar z} u_z +(-16Q\bar Q e^{-2u}+e^u)u_zu_{\bar z} \\
&\ge & \sfrac12 u_zu_{\bar z} + 4\bar Q e^{-2u} Q_z u_{\bar z} +
4Qe^{-2u} \bar Q_{\bar z} u_z
\end{eqnarray*}
since the assumption that $u$ is small implies $-16Q\bar
Qe^{-2u}+e^u\ge0$.  In coordinate-free notation, we see that at the
maximum point of $v=\frac14\|\na u\|^2$,
$$ 0 \ge \sfrac12\|\na u\|^2 + 4e^{-2u}\na(\|U\|^2) \cdot \na u \ge \sfrac12
\|\na u\|^2 - 2\epsilon \|\na u\|^2 -
\sfrac2{\epsilon}\|e^{-2u}\na(\|U\|^2)\|^2$$
for any $\epsilon>0$.
For $\epsilon=\frac18$, we see that at the maximum point $p$ of $v$,
that $v=\frac14\|\na u\|^2 \le 16\| e^{-2u}\na(\|U\|^2)\|^2$. Thus $v$
is bounded by the maximum value of $16\|e^{-2u}\na(\|U\|^2)\|^2$.

That $m\to 0$ as $U\to 0$ then follows from Corollary \ref{u-limit}
above. The explicit bound above shows $C\to0$ as $m\to 0$.
\end{proof}

\section{Representations of the Fundamental Group.} \label{rep-sec}
Fix a smooth compact oriented surface $S$ of genus at least two. By Theorems
\ref{exist-sol}
and \ref{thm:Legframe},
given a marked conformal structure $\Sigma$ on $S$ and small cubic
holomorphic
differential on $\Sigma$ we obtain, via the solution $u$ to \eqref{tz-eq},
a holonomy map
$$
\chi \!: \mathcal K \to {\rm Hom}(\pi_1 S,SU(2,1))/SU(2,1),\quad
(\Sigma,U)\mapsto [\rho],
$$
into the representation space of $\pi_1S$ in $SU(2,1)$. The domain of this
map is
$$\mathcal K = \{(\Sigma,U): \max_\Sigma \|U\|^2_h < \sfrac 1{54} \}$$
where $\Sigma$ ranges over all marked conformal structures on $S$, i.e.,
over all points in the
Teichm\"uller space of $S$. The norm $\|U\|_h$ is that
induced by the hyperbolic metric $h$ on $\Sigma$. As a manifold $\mathcal K$ is
a fibre subbundle
of the vector bundle over Teichm\" uller space whose fibre at $\Sigma$ is
the vector
space $H^0(\Sigma,K^3)$ of globally holomorphic cubic differentials.

A representation of $\pi_1 S$ into $SU(2,1)$ is called
$\re$-Fuchsian if it is discrete, faithful and conjugate to a
representation into $SO(2,1)$. In this case, it preserves a
Lagrangian plane.  The discussion in Section \ref{U-zero} above
shows that $\re$-Fuchsian representations correspond exactly to
pairs $(\Sigma,U)$ with cubic differential $U=0$.

An important invariant of representations of surface groups into
$SU(n,1)$ is the Toledo invariant \cite{toledo89}. Given an
invariant surface in $\CH^n$, the Toledo invariant is a normalised
integral of the pull-back of the K\"ahler form.  In the present
$n=2$ case, Xia has shown that the level sets of the Toledo
invariant are connected components of the representation space
\cite{xia00}.
\begin{prop}
The Toledo invariant vanishes for the representations we have
produced.
\end{prop}
\begin{proof}
For each such representation, we have constructed an equivariant
Lagrangian surface.
\end{proof}

\begin{thm} \label{local-diff}
The map $\chi$ is a local diffeomorphism near the zero section
$\{U=0\}$.
\end{thm}

The proof uses a symplectic form on the representation space due to Goldman
[7], which we pause to
describe. Let $G$ denote $SU(2,1)$, and let $\fg$ denote its Lie algebra.
Recall that the
representation space $\Hom(\pi_1S,G)/G$ is bijective to the moduli
space ${\mathcal M}_G$ of flat principal $G$-bundles over $S$. A point
$P\in{\mathcal M}_G$
is smooth if the centraliser of the image $\rho(\pi_1S)$ of the
corresponding representation has
dimension zero. At such a point
the tangent space can be identified with the cohomology
$H^1(S,\mathrm{ad}P)$, where
$\mathrm{ad}P$ is the associated flat $\fg$ bundle. The cohomology is de
Rham cohomology with
respect to the flat connexion $d_P$ on $\mathrm{ad}P$. Goldman's symplectic
form is defined as
follows. For any pair of $d_P$-closed $1$-forms
$Z,W\in\Omega^1_S(\mathrm{ad}P)$ representing
cohomology classes $[Z],[W]\in H^1(S,\mathrm{ad}P)$, he shows that
\[
\omega_P([Z],[W]) = \int_S\tr(Z\wedge W)
\]
is well-defined and symplectic at each smooth point $P$.

In particular, suppose $P$ is the flat principal bundle whose connexion is
determined by
the Maurer-Cartan $1$-form $\alpha$ in Theorem \ref{thm:Legframe}. When
$X\in
T_{(\Sigma,U)}\mathcal{K}$ is tangent to a curve $\gamma(t)$ in $\caK$ at
$t=0$,
its push-forward to $H^1(S,\mathrm{ad}P)$ is represented by the first
variation
$\delta_{X}\alpha$ of $\alpha$ along this curve, as a $\fg$-valued $1$-form
on $S$. Thus we have
\[
\chi^*\omega(X_1,X_2) =
\int_S\tr(\delta_{X_1}\alpha\wedge\delta_{X_2}\alpha).
\]

\begin{proof}
The proof proceeds by using Goldman's symplectic form and the Inverse Function Theorem.
First of all, any
solution for $U=0$ is an $\re$-Fuchsian representation, since it
corresponds to an embedding of $\re\mathbb H^2\subset \co\mathbb
H^2$.

Second, the representation space of $SU(2,1)$ near any
$\re$-Fuchsian representation is smooth and has dimension $16g-16$.
The smoothness follows from realizing that the holonomy representation
of an $\re$-Fuchsian representation in $SU(2,1)$ has zero centraliser (by
\cite{goldman84}). Moreover, the representation space is Hausdorff, since
$\rho(\pi_1S)$ is not contained in a parabolic subgroup \cite{johnson-millson87}.
The dimension is calculated in \cite{goldman84}.  Note
the Riemann-Roch Theorem shows that $\mathcal K$
has the same real dimension $16g-16$.

Third, at any point in $\mathcal K$ where $U=0$, we prove the tangent
map of $\chi$ is a linear isomorphism by showing that the pullback
$\chi^*\omega$ is nondegenerate. The tangent space $T_{(\Sigma,0)}
\mathcal K$ can be split into a Teichm\"uller space part and a fibre
part, and so each tangent vector can be split into a holomorphic
cubic differential $U$ plus a tangent vector to Teichm\"uller space,
which we may represent as a harmonic Beltrami differential $\mu$.

The nondegeneracy of $\chi^*\omega$ follows from the following three
claims, where $\delta_U\alpha$ represents the variation
$\left.\frac{\partial}{\partial t}\alpha(\Sigma,tU)\right|_{t=0}$.
\begin{itemize}
\item For any nonzero holomorphic cubic differential $U$,
$\omega(\delta_U\alpha,\delta_{iU}\alpha) \neq 0$.
\item If $U$ is a holomorphic cubic differential and $\mu$ is a
harmonic Beltrami differential, $\omega(\delta_U\alpha, \delta_\mu
\alpha) = 0$.
\item If $\mu,\nu$ are harmonic Beltrami differentials, then
$\omega(\delta_\mu\alpha,\delta_\nu\alpha)$ is a nonzero multiple of
the Weil-Petersson pairing ${\rm Im}\int_S \mu\cdot h\cdot \bar
\nu$, for $h$ the hyperbolic metric.
\end{itemize}
These claims show that the above splitting of $T_{(\Sigma,0)}
\mathcal K$ is a symplectic-orthogonal splitting of nondegenerate
spaces.

To prove the first claim, note that at $U=0$, the variation of the
metric $\delta_U s = 0$ (for the metric $2s^2|dz|^2$ above). This
follows since $U$ appears quadratically in (\ref{tz-eq}). Recall
$U=Q\,dz^3$.
$$\delta_U \alpha = \delta_U (A\,dz + B\,d\bar z) = \left(
\begin{array}{ccc} 0 & \bar Q s^{-2}d\bar z & 0 \\ - Q s^{-2}dz
&0 &0 \\ 0&0&0 \end{array} \right).$$ Then we may compute $${\rm
tr}\,(\delta_U \alpha \wedge \delta_{iU}\alpha) = 2i|Q|^2s^{-4}d z
\wedge d \bar z, $$ which is, up to a constant,
equal to
$|U|^2 h^{-2}$ for $U$ the cubic differential and $h$ the
hyperbolic metric. (Recall that we are varying around the zero cubic
differential.) Now since $|U|^2h^{-2}$ is naturally a section
of $|K|^2$ for $K$ the canonical bundle, it is a nonnegative
volume form, and thus $\omega(\delta_U\alpha,
\delta_{iU}\alpha)\neq0$.

For the second claim, note that for $U=0$, the deformation of the
connexion $\alpha$ in the direction of the harmonic Beltrami
differential $\mu$ is of the form
$$ \delta_\mu\alpha = \left(\begin{array}{ccc} * & 0 & * \\
0 & * & * \\ * & * & 0 \end{array} \right).$$ Therefore, ${\rm tr}\,
(\delta_U\alpha \wedge \delta_\mu \alpha) = 0$, and so
$\omega(\delta_U \alpha, \delta_\mu \alpha)=0$.

The third claim follows from a result of Shimura \cite{shimura59}
(see also Goldman \cite{goldman84}). If $U=0$, the holonomy of the
connexion $\alpha$ is contained in $SO(2,1) \subset SU(2,1)$, since
the developed surface is an $\mathbb {RH}^2\subset \mathbb {CH}^2$.
Therefore, ${\rm tr}\,(\delta_\mu \alpha \wedge \delta_\nu\alpha)$
is the same as the trace form on $SO(2,1)$. Under the Lie algebra
isomorphism $\mathfrak{so}(2,1)\sim \mathfrak{sl}(2,\mathbb R)$, the
trace forms on $SO(2,1)$ and $SL(2,\re)$ are the same up to a
nonzero constant multiple.  Then Shimura's result shows this trace
form is a multiple of the imaginary part of $\int_S \mu\cdot h\cdot
\bar \nu$.

Therefore, $\chi^*\omega$ is nondegenerate at $(\Sigma,0)\in
\mathcal K$, and so the Inverse Function Theorem shows $\chi$ is a
local diffeomorphism there.
\end{proof}
\begin{rem}It is likely that this computation can be pushed further to
show that $\chi$ is a local diffeomorphism away from $U=0$.  In
order to do this, we must have a good model of varying both $\Sigma$
and $U$ away from $U=0$ (and a direct verification, using
connexions, of a generalisation of Shimura's result).
\end{rem}
 We say a representation
$\rho\!: \pi_1 S\to SU(2,1)$ is \emph{geometrically finite} if
for $\Omega\subset
\partial \mathbb{CH}^2$ the domain of discontinuity of the action,
the quotient of $(\mathbb{CH}^2\cup\Omega)/\rho(\pi_1S)$ is a compact
manifold with boundary. (This definition should be modified in situations in which cusps are allowed.)

As in Parker-Platis \cite{parker-platis06}, we say a representation $\rho$ is
\emph{complex hyperbolic quasi-Fuchsian} if it is discrete,
faithful, geometrically finite and totally loxodromic. Recall that a
representation $\rho$ into $SU(2,1)$ is called \emph{totally
loxodromic} if every $\rho(\gamma)$ is loxodromic for $\gamma$ not
the identity. $\rho(\gamma)$ is \emph{loxodromic} if it fixes
exactly two points in $\partial\CH^2$.

\begin{thm} \label{main-thm}
There is a neighborhood $\mathcal N$ of the zero section $\{U=0\}$
of the total space of the vector bundle over Teichm\"uller space
whose fibre is the space of holomorphic cubic differentials so that
$$\chi|_\mathcal N : \mathcal N \to {\rm
Hom}(\pi_1S,SU(2,1))/SU(2,1)$$ is a diffeomorphism onto its image.
For each of these representations, there is a fundamental domain in
$\mathbb{CH}^2$, an equivariant minimal Lagrangian surface, and an
equivariant
submersion of $\CH^2$ onto the surface.
Each of these representations is complex hyperbolic quasi-Fuchsian.
\end{thm}
\begin{proof}
Restrict to cubic differentials $U$ so that $\sup \|U\|_h$ is small
enough; then Corollary \ref{fund-domain} and Theorem
\ref{local-diff} provide the bulk of the theorem.
All that remains is to show that the representations we produce are
complex hyperbolic quasi-Fuchsian.

The existence of the fundamental domain immediately implies the
representation is discrete and faithful. Geometric finiteness
follows from Corollary \ref{fund-domain}, in particular the fact
that $\bar\Phi$ is an  injective immersion of manifolds with
boundary.
We show $\rho$ is totally loxodromic below in Proposition \ref{tot-lox}.
\end{proof}

Let $\Gamma = \rho(\pi_1\Sigma)$ be the induced
discrete subgroup of $SU(2,1)$. Recall the limit set $\Lambda(\Gamma)$
is the subset of $\overline{\CH^2}$
defined by $$\Lambda(\Gamma) = \{ y = \lim_{i\to\infty}g_i(x) \,\,\big|\,\,
 x \in \overline{\CH^2},\, g_i\in\Gamma, \, g_i\neq g_j\mbox{ if }
i\neq j \}.$$ Our construction of the fundamental domain $\mathcal F$
shows the following lemma, whose proof follows immediately from Corollary
\ref{fund-domain}.
\begin{lem} \label{limit-set}
$\bar {\mathcal F} \cap \Lambda(\Gamma) = \emptyset$.
\end{lem}

\begin{prop} \label{tot-lox}
The representation $\rho$ is totally loxodromic.
\end{prop}
\begin{proof}
This is a standard fact, once we have our locally finite fundamental domain $\mathcal F$
(see e.g.\ \cite{parker-platis06}), but we provide a proof for the reader's
convenience.  We would like to thank Bill Goldman and especially John Parker for
explaining the essential ideas here to us.

We need only rule out elliptic and parabolic elements of $\Gamma\setminus\{1\}$.

Ruling out elliptic elements is straightforward. If $g\in\Gamma\setminus\{1\}$
fixes a point $p\in\CH^2$, then $p$ must lie in a translate $h\bar{\mathcal F}$
for some $h\in\Gamma$. But since $g$ has infinite order (as $\Gamma$ is a surface group),
that would imply that all $g^np \in h\bar{\mathcal F}$, which violates Lemma
\ref{limit-set}.

The remaining case is to rule out parabolic fixed points. Let $p$ be a fixed point
of a nontrivial parabolic element of $\Gamma$. Then $p\in\partial\CH^2$ and $p\in
\Lambda(\Gamma)$. There are analogues of the classical horoball construction, due to
Kamiya and
Parker \cite{kamiya83,kamiya91,parker97,kamiya-parker08}. Let $\Gamma_p$ denote the isotropy group of $p$.

The (modified) horoballs are
open sets $\mathcal B_\ell\subset \CH^2$ for $\ell\ge0$ satisfying
\begin{enumerate}
\item $p\in \overline {\mathcal B_\ell}$ for all $\ell\ge0$.
\item $\overline{\mathcal B_\ell} \subset \mathcal B_k \cup \{p\}$ if $\ell>k$.
\item $\cap_{\ell\ge0} \overline{\mathcal B_\ell} = \{p\}$.
\item For any $\ell\ge0$, $g\in\Gamma$, $g\mathcal B_\ell = \mathcal B_\ell$ if and only if $g\in\Gamma_p$.
\label{para-act-horo}
\item For any $\ell\ge0$, $g\in\Gamma$, $\overline{\mathcal B_\ell} \cap g\overline{\mathcal B_\ell} = \emptyset$
if and only if $g\notin \Gamma_p$.
\label{permute-horo}
\end{enumerate}

Now for $i=1,2,\dots$, choose $z_i \in \mathcal B_i$ so that $z_i\to p$. Then there are elements $g_i\in \Gamma$
so that $g_iz_i\in \bar{\mathcal F}$. By compactness, upon passing to a subsequence, we may
assume $g_iz_i\to q\in \bar {\mathcal F}$.  Now by properties (\ref{para-act-horo}) and
(\ref{permute-horo}), we have the following two cases

\noindent {\bf Case 1:} Upon passing to a subsequence, $\{g_i\mathcal B_1\}$ are disjoint.
In this case, we may assume that the Euclidean volume (as measured in $\re^4 = \co^2\supset\CH^2$)
of $g_i\mathcal B_1$ goes to zero. This implies that the Euclidean diameter
of $g_i\mathcal B_1$ also goes to zero. Therefore,
$g_iz_i\to q$ implies $g_ip\to q$. Thus $q$ is a limit
point of $\Gamma$, which contradicts Lemma \ref{limit-set}.  (The assertion
about the relationship between the Euclidean volume and diameter is valid
for all sufficiently small domains in $\CH^2\subset\C^2$, and may be checked
infinitesimally by calculating the Jacobian matrix of the action of a general element of $SU(2,1)$
in inhomogeneous projective coordinates.)

\noindent{\bf Case 2:} Upon passing to a subsequence, all $g_i\mathcal B_1 = g_1\mathcal B_1$.
In this case, $g_1^{-1}g_i\in \Gamma_p$, and so $g_1^{-1}g_i\mathcal B_i = \mathcal B_i$.
Therefore, $g_iz_i \in g_1\mathcal B_i$, and taking $i\to\infty$ shows that $q=g_1p$. But $g_1p$ is a
limit point of $\Gamma$, which again contradicts Lemma \ref{limit-set}.
\end{proof}

\def\cprime{$'$}


\end{document}